\documentclass[reqno]{amsart}
\usepackage[foot]{amsaddr}

\usepackage{amssymb,amsmath,amsthm,amstext,amsfonts}
\usepackage[dvips]{graphicx}
\usepackage{psfrag}
\usepackage{url}
\usepackage{amsmath,amstext,amsthm,amsfonts}
\usepackage{color}
\usepackage{xcolor}

\usepackage[colorlinks=true, linkcolor=blue, urlcolor=red, citecolor=blue, hyperindex, backref]{hyperref}

\makeatletter \@addtoreset{equation}{section} \makeatother

\renewcommand\thetable{\thesection.\@arabic\c@table}

\theoremstyle{plain}
\newtheorem{maintheorem}{Theorem}

\newtheorem{mainconjecture}{Conjecture}

\newtheorem{maincorollary}{Corollary}

\newtheorem{mainproblem}{Problem}

\newtheorem{mainquestion}{Question}

\newtheorem{theorem}{Theorem }[section]
\newtheorem{proposition}[theorem]{Proposition}
\newtheorem{lemma}[theorem]{Lemma}
\newtheorem{corollary}[theorem]{Corollary}

\theoremstyle{definition} \theoremstyle{remark}
\newtheorem{remark}[theorem]{Remark}
\newtheorem{example}[theorem]{Example}
\newtheorem{definition}[theorem]{Definition}

\newcommand{\Sc}{\mathbb{S}}
\newcommand{\Lo}{\mathcal{L}}

\newcommand{\C}{\mathbb{C}}
\newcommand{\Z}{\mathbb{Z}}
\newcommand{\N}{\mathbb{N}}
\newcommand{\R}{\mathbb{R}}

\newcommand{\Leb}{\operatorname{Leb}}

\begin{document}

\title{Thermodynamical and spectral phase transition for local diffeomorphisms in the circle}

\author{ Thiago Bomfim$^{\ast}$ and Victor Carneiro}

\address{Departamento de Matem\'atica, Universidade Federal da Bahia\\
Av. Ademar de Barros s/n, 40170-110 Salvador, Brazil.}
\email{$^{\ast}$Corresponding Author: tbnunes@ufba.br}
\email{victor.carneiro93@gmail.com}

\date{\today}

\begin{abstract}
It is known that all uniformly expanding dynamics have no phase transition with respect to H\"older continuous potentials. In this paper we show that given a local diffeomorphism $f$ on the circle, that is neither a uniformly expanding dynamics nor invertible, the topological pressure function $\mathbb{R} \ni t \mapsto P_{top}(f , -t\log |Df|)$ is not analytical. In other words,  $f$ has a thermodynamic phase transition with respect to geometric potential. Assuming that $f$ is transitive and that $Df$ is H\"older continuous, we show that there exists $ t_{0} \in (0 , 1]$ such that the transfer operator $\mathcal{L}_{f, -t\log|Df|}$, acting on the space of H\"older continuous functions, has the spectral gap property for all $t < t_{0}$ and has not the spectral gap property for all $t \geq t_{0}$. Similar results are also obtained when the transfer operator acts on the space of bounded variations functions and smooth functions. In particular, we show that in the transitive case $f$ has a unique thermodynamic phase transition and it occurs in $t_{0}$. In addition, if the loss of expansion of the dynamics occurs because of an indifferent fixed point or the dynamics admits an absolutely continuous invariant probability with  positive Lyapunov exponent then $t_0 = 1.$

\end{abstract}


\maketitle


\section{Introduction}

This paper is concerned with the phase transition problem of smooth dynamical systems. From a statistical mechanics point of view, phase transition means a dramatic change on a given medium after the action of certain external conditions. In the pioneering works of Sinai, Ruelle and Bowen \cite{Rue68, S72, Bow75} the thermodynamical formalism was brought from statistical mechanics to dynamical systems. From a dynamical systems point of view, phase transitions are often associated to the lack  of differentiability or analyticity of the topological pressure, as a function of the potential.

 More precisely, given $f : \Lambda \rightarrow \Lambda$  a continuous dynamical system on a compact metric space $\Lambda$  and $\phi : \Lambda \rightarrow \R$ a continuous potential, the variational principle for the pressure asserts that 
$$
P_{top}(f, \phi) = \sup\{h_{\mu}(f) + \int\phi d\mu : \mu \text{ is an } f-\text{invariant probability}\}
$$
where $P_{top}(f, \phi)$ denotes the topological pressure of $f$ with respect to $\phi$ and $h_{\mu}(f)$ denotes the Kolmogorov-Sinai's metric entropy.  An equilibrium state $\mu_{\phi}$ for $f$ with respect to $\phi$ is a probability measure that attains
the supremum. In the special case that $\phi \equiv 0$, the topological pressure $P_{top}(f, \phi)$ coincides with  the topological entropy  $h_{top}(f)$, which is one of the most important topological invariants in dynamical systems (see e.g. \cite{W82}). 
Thereby, we say that $f$ has a phase transition with respect to $\phi$ if the topological pressure function
$$
\R \ni t \mapsto P_{top}(f, t\phi)
$$
is not analytic.
It is known from Classical Ergodic Theory that if $f$ has finite topological entropy then the pressure function $\R \ni t \mapsto P_{top}(f, t\phi)$ is convex and Lipschtiz continuous. Then, it follows from Rademacher's theorem that $\R \ni t \mapsto P_{top}(f, t\phi)$ is differentiable except in at most a countable number of points.
By \cite{W92}, if $f$ is expansive and has finite entropy then the lack of differentiability  of the topological pressure function is related to non uniqueness of the equilibrium states associated to potential $\phi$ (for generalizations of this result see e.g. \cite{BCMV20, IT20}). 

The notion of uniform hyperbolicity, that was introduced in the seventies by Smale \cite{Sm67}, drew the attention of many researchers and influenced the way of studying differentiable dynamical systems. In the context of non-invertible maps, hyperbolic dynamics correspond to hyperbolic endomorphism (see \cite{P76}) and uniformly expanding maps.
It follows from \cite{Rue68, S72, Bow75} that if $f$ is a transitive hyperbolic or expanding dynamics and $\phi$ is H\"older continuous then $f$ doesn't have phase transition with respect to $\phi$. On the other hand, several non-uniformly hyperbolic examples are know to present phase transitions with respect to regular potentials. For example:
\begin{itemize}
\item Manneville-Pomeau maps and geometric potential \cite{Lo93}, 
\item large class of maps of the interval with indifferent fixed point and geometric potential \cite{PS92},
\item certain quadratic maps and geometric potential \cite{CRL13}, 
\item certain non-degenerate smooth interval maps and geometric potential \cite{CRL21}, 
\item porcupine horseshoes and geometric potential \cite{DGR14},
\item geodesic flow on Riemannian non-compact manifolds with variable pinched negative sectional curvature and suitable H\"older continuous potential \cite{IRV18}, 
\item geodesic flow on certain M-puncture sphere and geometric potential \cite{V17}. 
\end{itemize}
However, a characterization of the set of dynamics which have phase transitions seems still out of reach.

It is important to emphasize that when the dynamics is restricted, for example a shift or suspension flows  defined over countable alphabets, it is necessary that in the phase transition problem the potentials have low regularity (see \cite{S01,S06,CL15,L15,IJ13}).

In view of this discussion, we propose the following problem:

\begin{mainproblem}\label{probA}
Given $f$ a local diffeomorphism on a compact manifold, with positive topological entropy, and $\phi$ a H\"older continuous potential, can we characterize when $f$ has a thermodynamical phase transition with respect to $\phi$ ?
\end{mainproblem}

The entropy requirement above is to assume some chaoticity. Indeed, if $f$ is the irrational rotation on the circle then $f$ is not hyperbolic or expanding, $h_{top}(f) = 0$ however $\R \ni t \mapsto P_{top}(f , t\phi) = t \int \phi d m$, where $m$ denotes the Lebesgue measure in the circle. In particular $f$ has not phase transition for any continuous potential.

When $f$ is a mixing expanding or hyperbolic dynamics and $\phi$ is a suitable potential  the thermodynamical properties can be recovered through the Ruelle-Perron-Frobenius operator or transfer operator $\mathcal{L}_{\phi}$ acting on functions $g : M \rightarrow \C$ defined as the following:
 $$
 \mathcal{L}_{f,\phi}(g)(x) := \sum_{f(y) = x}e^{\phi(y)}g(y).
 $$
 Indeed, using the fact that $\mathcal{L}_{\phi}$ has the spectral gap property (see precise definition in Section \ref{Statement of the main results}) acting on a suitable Banach space\footnote{If $M$ is one dimensional manifold then usually the suitable Banach space is for example the bounded variations space and in the general case the H\"older continuous spaces, smooth function spaces, distributions spaces, among others.} it is possible to show that $f$ has not phase transitions with respect to suitable potentials (see by e.g. \cite{PU10}).

It is important to note that the spectral gap property is also useful to the study of finer statistical properties of thermodynamical quantities as equilibrium states, mixing properties, large deviation and limit theorems, stability of the topological pressure and equilibrium states or differentiability results for thermodynamical quantities (see e.g. \cite{Ba00,GL06, BCV16,BC19}).

Furthermore, in \cite{PS92} it shown that for a large class of maps $f$ which are piecewise monotone on interval with an indifferent fixed point and the geometric potential $\phi := -\log |Df|$ the associated transfer operator  has not the spectral gap property when acting on the space of bounded variation functions.
 
 In view of this discussion, we propose a conjecture:
 
 \begin{mainconjecture}\label{conjA}
Let $f : M\rightarrow M$ be a $C^{2}-$local diffeomorphism on a Riemannian manifold $M$. If $f$ is not a uniformly expanding dynamic then there exists a H\"older continuous potential $\phi$ such that $\mathcal{L}_{f, \phi}$ does not have the spectral gap property acting on a H\"older continuous function space.
\end{mainconjecture}

 In \cite{GKLM18}, under the hypothesis that \emph{for every} potential $\phi$ the transfer operator $\mathcal{L}_{f, \phi}$ has the spectral gap property acting on some suitable Banach space, it is studied the geometry of the set of potentials and consequences for the understanding of the thermodynamics quantities of $f$. Continuing this work in \cite{LR20}, it is studied the curvature of the equilibrium states  space for  a complete shift.
 
On the other hand, it follows from \cite{Kl20} that for any map of the circle which is expanding outside an arbitrarily flat neutral point, the set of H\"older potentials such that $\mathcal{L}_{f, \phi}$ exhibits the spectral gap property on the H\"older functions is dense in the uniform topology.  A natural question is what happens  for  one parameter families of potentials, in particular families induced by the geometric potential, which allow us to build absolutely continuous invariant probabilities with respect to Lebesgue measure (a.c.i.p.) and physical measures:

\begin{mainquestion}
If $f : \Sc^{1} \rightarrow \Sc^{1}$ is a  $C^{1}-$local diffeomorphism in the circle, with  positive topological entropy and $Df$ is H\"older continuous, then the set
 $$\big\{t \in [0 , 1] \;;\; \mathcal{L}_{f, -t\log |Df|} \text{ has the spectral gap property}
 $$
 $$
 \text{on a H\"older continuous function space}\big\}$$ is a dense subset of $[0 , 1]$?
 \end{mainquestion}

Remember that for a large class of dynamics of the interval with indifferent fixed point, in \cite{PS92}, is proved that except for $t\geq1$ we have that $\mathcal{L}_{f, -t\log |Df|}$ has the spectral gap property on the space of functions with bounded variations.

In \cite{CS09}, spectral gap property is investigate for countable shifts, and the following
problem is proposed:\\

\emph{When does a potential $\phi$ satisfy the SGP? How common is this phenomenon? What are the most important obstructions?}

\

It is in fact proved that SGP holds densely and a description is given to when it holds.

As a first step in studying these issues for smooth maps, in this paper we obtained answers to the conjecture and question for dynamics in the circle.

This paper is organized as follows. In Section~\ref{Statement of the main results} we provide some definitions
and the statement of  the main results on thermodynamical and spectral phase transition. 
In Section \ref{stratpro} are summarized the key ideas of the proofs of each main result.
In Section~\ref{prelim} we recall the necessary framework on Topological Dynamics, Thermodynamical formalism and Transfer operator. In Section~\ref{Proofs} the main results are proved. Finally, in the Section~\ref{disque}  some explicitly examples are given, discussions on the main results are realized and additional questions are proposed. In particular; in  Section \ref{t0} is calculated the phase transition parameter, in Section \ref{nontra} some comments about the non-transitive case are made and Section \ref{hdc} it is discussed briefly on the higher dimensional context.
%
%

\section{Statement of the main results}\label{Statement of the main results}

This section is devoted to the statement of the main results.

Throughout the paper we shall denote the circle $\{z \in \C : |z| = 1\}$ by $\Sc^1$.

Our first result ensures thermodynamical phase transition. In fact, the phase transition occurs  with respect to geometric potential.

\begin{maintheorem}\label{mainthA}
Let $f:\Sc^1 \to \Sc^1$  be a non invertible $C^1$ local diffeomorphism. If $f$ is not an expanding dynamics  then $f$ has phase transition with respect to $-\log|Df|$.
\end{maintheorem}

In fact;  if $f$ has any negative Lyapunov exponent, we show that the pressure function is not even differentiable at one point.

\begin{remark}\label{remar21}
Since $f:\Sc^1 \to \Sc^1$ is a local diffeomorphism, it follows from Classical Ergodic Theory and \cite{MP77} that $h_{top}(f) = \log \deg(f)$, where $deg(f)$ is the topological degree of $f$. In particular, $f$ is not invertible if and only if $h_{top}(f) > 0$.
Note also that if $f$ has zero topological entropy it's easy to see that given a continuous potential $\phi$ the topological pressure function $\R \ni t \mapsto P_{top}(f , t\phi)$ is linear in each of the intervals $(-\infty,0]$ and $[0,\infty)$.
\end{remark}

The previous theorem was expected and its proof is not difficult, however it's the starting point for the finer understanding obtained in the next results.

Given $E$ a complex Banach space and $T : E \rightarrow E$ a bounded linear operator, we say that $T$ has the \emph{(strong) spectral gap property} if there exists a decomposition of its spectrum $sp(T) \subset \C$ as follows: 
$
sp(T) = \{\lambda_{1}\} \cup \Sigma_{1}$ where $\lambda_{1} >0$ is a leading eigenvalue for $T$ with one-dimensional associated eigenspace and there exists $0 < \lambda_{0} < \lambda_{1}$ such that $\Sigma_{1} \subset \{z \in \C : |z| < \lambda_{0}\}$.

Given $r \geq 1$ an integer and $\alpha \in (0 , 1]$ we denote by $C^{r}(\Sc^1 , \C)$ and $C^{\alpha}(\Sc^1 , \C)$ the Banach spaces of $C^{r}$ functions  and $\alpha-$H\"older continuous complex functions whose domain is $\Sc^1$, respectively. Furthermore, we denote by $BV[\Sc^1]$ the Banach space of bounded variations complex functions on $\Sc^1$.

Our second result ensures, assuming that the dynamics is transitive, an effective spectral phase transition.

\begin{maintheorem}\label{mainthC}	
Let $E = C^{\alpha}(\Sc^{1} , \C)$ or $C^{r}(\Sc^{1} , \C)$  and let $f:\Sc^1 \to \Sc^1$ be a transitive non invertible $C^{1}-$local diffeomorphism with $Df \in E$. If $f$ is not an expanding  dynamics then there exists $t_{0} \in (0 , 1]$ such that:

(i) the transfer operator $\mathcal{L}_{f, -t\log|Df|}$ has the spectral gap property on $E$ for all $t < t_{0}$ and has not the spectral gap property for all $t \geq t_{0}$.

(ii) If $E = BV[\Sc^{1}]$  then $\mathcal{L}_{f, -t\log|Df|}$ has the spectral gap property on $BV[\Sc^{1}]$ for all $t \in [0 ,  t_{0})$ and has not the spectral gap property for all $t \geq t_{0}$.
\end{maintheorem}

Note that in the context of Theorem \ref{mainthC}, $f$ is piecewise monotone with \textit{full branches} and transitive. Hence, it follows directly from \cite[Corollary 4.3]{CM86} that $f$ is conjugate to an expanding dynamic.

\begin{remark}
The previous theorem is not direct consequence of \cite{PS92} because they assume that the loss of expansion is only caused by an indifferent fixed point similar to a Maneville-Poumeau map and $E = BV[\Sc^{1}]$. In this sense, the previous theorem can be seen as an extension of \cite{PS92}.
\end{remark}

As a consequence of the previous theorem we prove an effective thermodynamical phase transition. 

\begin{maincorollary}\label{mainthD}	
Let $f:\Sc^1 \to \Sc^1$ be a transitive non invertible $C^{1}-$local diffeomorphism  with $Df$ H\"older continuous. If $f$ is not an expanding  dynamics then the topological pressure function $\R \ni t \mapsto P_{top}(f , -t\log |Df|) $ is analytical, strictly decreasing and strictly convex in $(-\infty , t_{0})$ and constant equal to zero in $[t_{0} , +\infty)$ (see figure \ref{graphic P}).
\end{maincorollary}

\begin{figure}[htb]
\includegraphics[width=8.1cm, height=4.5cm]{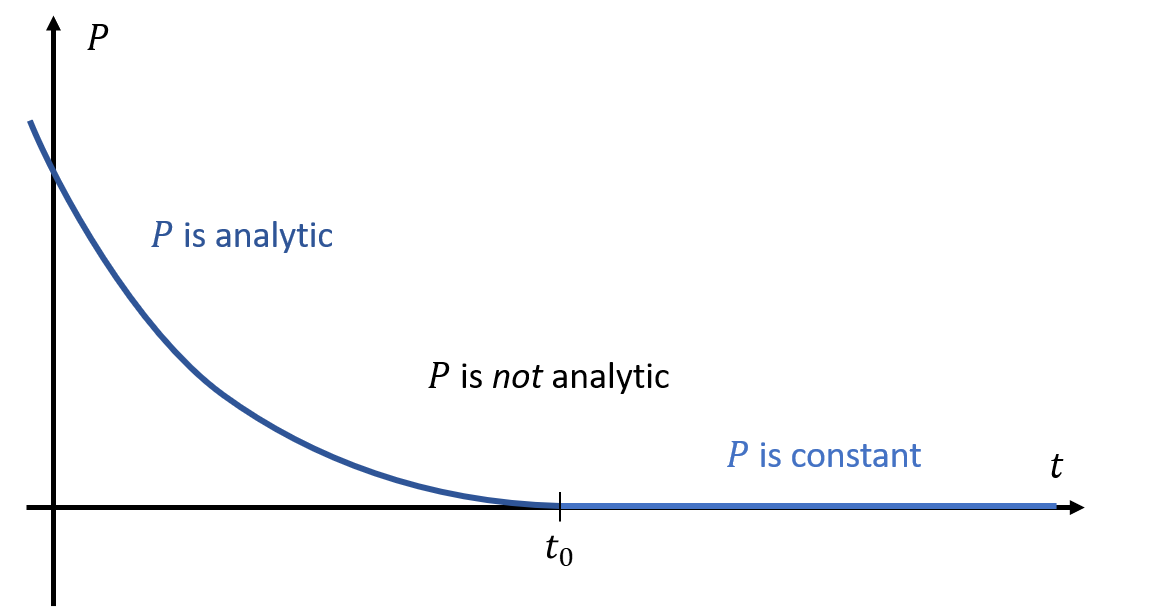}
\caption{Graphic of the topological pressure function}
\label{graphic P}
\centering
\end{figure}

We show that the $t_{0}$ in the previous theorem and corollary are the same. Moreover, $t_{0} = \inf\{ t \in \R : P_{top}(f , -t\log |Df|) = 0\}$ and it can be calculated using the Haursdorff dimension of the support of invariant probabilities with  positive Lyapunov exponents. In particular, if the loss of expansion is only caused by an indifferent fixed point or the dynamics admits an a.c.i.p. with  positive Lyapunov exponent then $t_{0} = 1$. (see details in Section \ref{t0}) 

Taking into consideration a nonwandering decomposition theorem, we see that the non transitive case can be reduced to a transitive context, however it is not possible to apply directly the Theorem \ref{mainthC}	(see details in Section \ref{nontra}).

\section{Strategy of the proofs}\label{stratpro}

In this section we summarize the key ideas behind each proof of the main results.

\subsection{Theorem \ref{mainthA}} We investigate the pressure function associated to probabilities with non-negative and non-positive Lyapunov exponents, which we denote $P_+$ and $P_-$, respectively. The pressure function is always the maximum of these two, with $P_+$ necessarily dominating before of $t=0$ and then $P_-$ after $t=1$. Thus the point of transition, which lies between $0$ and $1$, is exactly where phase transition occurs, since each of these functions must have non-negative and non-positive derivatives, respectively. The characterization of expanding maps via positive Lyapunov exponents \cite{CLR03} implies that an ergodic measure with zero or negative  Lyapunov exponent must exist, thus $P_-$ is well defined and transition occurs.

\subsection{Theorem \ref{mainthC}} The proof will be divided into three steps:

\begin{itemize}
\item In the first step we will show that in the thermodynamical phase transition parameter $t_{0}$, obtained in the Theorem \ref{mainthA}, the transfer operator $\mathcal{L}_{f, -t\log|Df|}$ has not the spectral gap property acting on E. For that step, following the ideas of the proof for expanding maps (such as in e.g. \cite{OV16}), we use Rokhlin's formula and prove that the invariant measure obtained via spectral gap property is an equilibrium state with topological pressure related to the transfer operator. Furthermore spectral theory ensures that the leading isolated simple eigenvalue varies analytically with respect to the transfer operator. 
\item In the second step we will show that, after the parameter $t_{0}$, the transfer operator $\mathcal{L}_{f, -t\log|Df|}$ has not the spectral gap property acting on $E$. For that step, assuming there is spectral gap for some parameter, we make use of Nagaev's method for obtaining the Central Limit Theorem to ensure that in the neighborhood of this parameter the geometric pressure function is strictly convex. In particular, the constant behaviour of the geometric pressure function after the transition parameter implies that spectral gap does not occur for these parameters.
\item In the third step we will show that before the parameter $t_{0}$ the transfer operator $\mathcal{L}_{f, -t\log|Df|}$ has the spectral gap property acting on $E$. For that step we prove that, in our context, quasi-compactness is sufficient for the transfer operator to have spectral gap. For this operator acting on bounded variation functions, we use the estimates on spectral radius from \cite{BK90} and \cite{Ba00} to prove quasi-compactness directly. For H\"older and smooth functions, we use estimates from \cite{BJL96} and \cite{CL97} to prove the essential spectral radius is bounded above by a translation of the pressure function. Using this estimate, we prove spectral gap for $t=0$, and therefore to  $t$ sufficiently close to zero, by openness of spectral gap. Next we extend this property to all $t<t_0$ using the monotonicity of the pressure function and the fact that it is related with the spectral radius, not only when the gap holds but also on the boundary of that region.
\end{itemize}

\subsection{Corollary \ref{mainthD}} The analyticity of the topological pressure function follows directly from Theorem \ref{mainthC} by analyticity of the leading isolated eigenvalue and its relation with the topological pressure.


\section{Preliminary}\label{prelim}

In this section we provide some definitions and preparatory results needed for the proof of the main results. 

\subsection{Topological Dynamics}

Given a continuous map $f:X\to X$ on a compact metric space $X$, we start defining some properties from a topological point of view describing the mixing behaviour of the system, that is, how much the orbits of points, or even open sets, tend to mix with each other and visit all regions of the space $X$. The basic expected type of mixing is called \textit{transitivity} and is defined as:

\begin{definition}
The dynamic $f$ is \textbf{transitive} if for every open sets $U$ and $V$, there exists $n\in\N$ such that $f^n(U)\cap V \neq 0$ or, equivalently, $f$ admits a dense orbit.
\end{definition}

Other concepts of mixing behaviour arise naturally, such as \textit{strong transitivity}, \textit{mixing}, all the way up to \textit{topological exactness}, also called \textit{locally eventually onto}:

\begin{definition}
The dynamic $f$ is \textbf{topologically exact} if for every open set $U$ there exists $n\in\N$ such that
$f^n(U)=X$.
\end{definition}

Evidently, topologically exactness implies transitivity, the converse need not be true. A consequence of topological exactness is that every set of pre-orbits $\bigcup_{n\geq 0}f^{-n}\{x\}$ is uniformly dense on $X$ for every $x\in X$.

Regarding the distances of orbits, two concepts are essential. A map is said to be \textit{expanding} if the distance of two sufficiently close points are expanded by $f$:

\begin{definition}
The dynamic $f$ is called \textbf{expanding} if $f$ is open map and there are constants $\sigma > 1$, $r>0$ and $n \geq 0$ such that 
$$d(f^{n}(x),f^{n}(y))\geq \sigma d(x,y) \text{ for all } d(x,y) < r. $$
\end{definition}

Note that if $X$ is a compact Riemannian manifold and $f$ is a local diffeomorphism then: $f$ is an expanding dynamics if, only if, $\inf_{x \in X}||Df(x)|| > 1$.

A dynamic is said to be \textit{expansive} if the orbits of any two points become distinguishable after enough time:

\begin{definition}
We say that $f$ is (positively) \textbf{expansive} if there is a constant $\varepsilon_0 >0$, called the \textbf{expansivity constant}, such that for every pair of points $x\neq y$ in $X$ there is $n\in \N$ such that d$(f^n(x),f^n(y))\geq \varepsilon_0$.
\end{definition}

Every expanding dynamics is expansive. In fact, every expanding dynamics defined on a connected domain will be expansive and topologically exact.

An example of expanding dynamics are the $d$-adic transformations or \textbf{Bernoulli maps}: for a given $d\in N$ define 
$$
f_d: \Sc^1 \xrightarrow[ x \mapsto dx \mod 1]{} \Sc^1.
$$
Remember that two maps $f: X \to X$ and $g: Y \to Y$ are \textbf{conjugate} if there exists a homeomorphism $h:X\to Y$ such that $g\circ h=h \circ f$.
Expansiveness, topological exactness and transitivity are in fact \textbf{topologically invariant}, that is, if a map is conjugate to another map with either of these properties then it also does. However, a conjugate of an expanding map need not be expanding as well. An example of that are the Manneville-Pomeau maps $f_\alpha:[0,1]\to[0,1]$:

\begin{align*}
   f_\alpha(x)=\begin{cases}
  x(1+2^\alpha x^\alpha), \text{ if } x \in [0,1/2]\\
  2x-1, \text{ if } x \in (1/2,1]
  \end{cases} 
\end{align*}
which are topologically exact, expansive but not expanding maps. Since $f_\alpha(0)=0$ is an \textit{indifferent fixed point}, that is $Df_\alpha(0)=1$.

\

\subsection{Ergodic theory}
Now we state some classical definitions,
 notations and results from Ergodic Theory (for more details see e.g. \cite{OV16}).
 
One fundamental result is the called \textbf{Birkhoff's Ergodic Theorem} which relates time and space averages of a given potential $\phi: X \to \R.$ 

\begin{theorem}(Birkhoff) Let $f:X\to X$ be a measurable transformation and $\mu$ be an $f$-invariant probability. Given any integrable function $\phi: X \to \R$, the limit:
$$\bar{\phi}(x)=\lim_{n\to\infty} \dfrac{1}{n}\sum_{j=0}^{n-1} \phi(f^j(x))$$
exists in $\mu$-a.e. $x \in X$. Furthermore, the function $\Bar{\phi}$ defined this way is integrable and satisfies
$$\int\bar{\phi}(x)d\mu(x)=\int\phi(x)d\mu (x).$$
Additionally, if $\mu$ is $f-$ergodic, then $\Bar{\phi}\equiv\int \phi d\mu$ for $\mu-$a.e..
\end{theorem}

We denote the $f-$invariant probabilities space by $\mathcal{M}_{1}(f)$  and  the $f-$invariant and ergodic probabilities space by $\mathcal{M}_{e}(f)$.

An important example of time average are the \textbf{Lyapunov exponents} which translate the asymptotical rates of expansion and contraction of a map. On a broader context, these are defined via the Oseledets multiplicative ergodic theorem. For our context, smooth local diffeomorphisms on the circle, the Lyapunov exponents are defined simply as:
$$\lambda(x)=\lim_{n\to\infty} \dfrac{1}{n}\log|Df^{n}(x))|=\lim_{n\to\infty}\sum_{j=0}^{n-1}\dfrac{1}{n}\log|Df(f^j (x))|,$$
whether the limit exists.
That is, the Lyapunov exponents $\lambda$ coincide with the time average for the continuous potential $\log|Df|$  in each point $x$ where the limit exists.
On the other hand, given an ergodic measure $\mu$, by Birkhoff's Ergodic Theorem, we have
$$\lambda(x)=\int \log|Df|d\mu \text{ for } \mu-\text{a.e. } $$ and we define the Lyapunov exponent for this measure $\chi_\mu(f):=\int \log|Df|d\mu$. Expanding dynamics are characterized by having strict positive Lyapunov exponents and that will play a significant role later on (see \cite{CLR03}).

Now, we present two estimates for the metric entropy that will be crucial later on the proof of main results. First the relation between entropy and Lyapunov exponents, that in our context of smooth local diffeomorphisms on the circle, guarantees us:

\begin{theorem}(Margulis-Ruelle inequality, \cite{Rue78})
Let $f : \Sc^{1} \rightarrow \Sc^{1}$ be a $C^{1}$-local diffeomorphism that preserves an $f-$invariant and ergodic probability $\mu$. Then $$h_\mu(f)\leq \max\{0 , \chi_{\mu}(f)\}.$$
\end{theorem}

Next, we study the relation between entropy and the \textit{Jacobian}.
Let $f$ be a local homeomorphism and a given probability $\nu$ (not necessarily invariant), define the \textbf{Jacobian} of $f$ with respect to $\nu$ as the measurable function $J_\nu(f)$, which is essentially unique, satisfying:
$$\nu(f(A))=\int_A J_{\nu}(f) d\nu$$
for any measurable invertibility domain $A$. Up to restricting $f$ to a full measure subset, $f$ always admits a Jacobian with respect to $\mu$, an invariant measure.

When we have a partition that generates the Borelian $\sigma$-algebra, we have a more direct way of obtaining the entropy of $f$ with respect to the measure $\mu$ using that partition.
 Given two partitions $\mathcal{P}$ and $\mathcal{Q}$ of $X$, we define a new partition $\mathcal{P}\vee \mathcal{Q}$ as the following:
$$\mathcal{P}\vee\mathcal{Q}:=\{A\cap B; A\in\mathcal{P} \text{ and } B \in \mathcal{Q}\}.$$
 
\begin{definition}
Given $f : X \rightarrow X$ a measurable map, we say that a finite partition $\mathcal{P}$ of $X$ is a \textbf{generating partition} if $\bigvee_{i=m}^{+\infty} f^{-i}(\mathcal{P})$ generates the $\sigma$-algebra of $X$, with $m=-\infty$ or $m=0$ for invertible and non invertible systems, respectively.
\end{definition} 

If additionally $X$ is a metric space, then a partition such that the diameter of the elements of $\bigvee_{i=m}^{+\infty} f^{-i}(\mathcal{P})$ gets arbitrarily small is a generating partition. Note that in our context of local diffeomorphisms on the circle, we will see, that the domains of the full branches form a generating partition. Any such partition gives the following way of calculating the entropy:

\begin{theorem}(Rokhlin's formula)\label{rokf} Let $f: M \to M$ be a locally invertible transformation,  on a separable metric space, and $\mu$ be an $f$-invariant probability. Assume that there is some generating, up to zero measure, partition $\mathcal{P}$ of $X$ such that every $P \in \mathcal{P}$ is an invertibility domain of $f$. Then

$$h_\mu(f)=\int \log J_\mu(f) d\mu.$$

\end{theorem}

This equality will be fundamental, later on the proof of Theorem \ref{mainthC}.

\

\subsection{Transfer Operator}

In this section  we recall some properties of the transfer operators. For more details on the transfer operator see e.g. \cite{S12} or \cite{PU10}.

In what follows, given $T : E \rightarrow E$ a bounded linear operator we denote its spectral radius by $\rho(T)$.

One of the main tools to study thermodynamical quantities  and indeed obtain equilibrium states, as well as its properties, is the \textit{Ruelle-Perron-Frobenius operator} or \textit{transfer operator}, which acts on function spaces:

\begin{definition}
Let $f : M \rightarrow M$ be a local homeomorphism on a compact and connected manifold. Given a complex continuous function $\phi: M \rightarrow \C$ , define the Ruelle-Perron-Frobenius operator or transfer operator $\mathcal{L}_{f, \phi}$ acting on functions $g : M \rightarrow \C$ this way:
 $$
 \mathcal{L}_{f,\phi}(g)(x) := \sum_{f(y) = x}e^{\phi(y)}g(y). \footnote{When no confusion is possible, for notational simplicity we omit the dependence of the transfer operator on $f$.}
 $$
\end{definition}

Classical thermodynamical results for sufficiently \textit{chaotic} dynamics, derive from good \textit{spectral} properties from this operator. 

If $\phi$ is a real continuous function, via Mazur's Separation Theorem, since $\Lo_{f,\phi}$ is a positive operator then it has $\rho(\Lo_{f,\phi}|_{C^0})$ as an eigenvalue for its dual operator, that is, there exists a probability $\nu_\phi$ with $(\Lo_{f,\phi}|_{C^{0}})^* \nu_{\phi} = \rho(\Lo_{f,\phi}|_{C^0}) \nu_\phi$.
If additionally $E \subset C^{0}(M , \C)$ is a Banach space continuously immersed in $C^{0}(M , \C)$ and $\Lo_{f,\phi}|_E$ has the spectral gap property, then $\rho(\Lo_{f,\phi}|_{C^0}) = \rho(\Lo_{f,\phi}|_{E})$ and $\Lo_{f,\phi}|_E$ admits an eigenfunction $h_\phi \in E$ with respect to $\rho(\Lo_{f,\phi}|_E)$ which is the leading eigenvalue. We can assume, up to rescaling, that  $\int h_{\phi}d\nu_{\phi} = 1$.
Then the probability $\mu_\phi=h_\phi \cdot \nu_\phi$ is proved to be $f-$invariant and is a candidate for the equilibrium state.

\begin{remark}
We recall the notion of analyticity for functions on Banach spaces. Let $E_{1} , E_{2}$ be Banach spaces and denote by $\mathcal{L}^{i}_{s}(E_{1} , E_{2})$ the space of symmetric $i$-linear transformations from
$E_{1}^{i}$ to $E_{2}$. For notational simplicity, given $P_{i} \in \mathcal{L}^{i}_{s}(E_{1} , E_{2})$ and $h\in E_{1}$ we set
$P_{i}(h) := P_{i}(h,\ldots,h)$.

We say the function $f : U \subset E_{1} \rightarrow E_{2}$, defined on an open subset,
is \emph{analytic} if for all $x \in U$ there exists $r > 0$ and for each $i\ge 1$ there exists $P_{i} \in \mathcal{L}^{i}_{s}(E_{1} , E_{2})$ (depending on $x$) such that
$$
f(x + h) = f(x) + \sum_{i=1}^{\infty}\frac{P_{i}(h)}{i!}
$$ for all $h \in B(0 , r)$ and the convergence is uniform.

Analytic functions on Banach spaces have completely similar properties to real analytic and complex analytic functions.
For instance, if $f : U \subset E \rightarrow F$ is analytic then $f$ is $C^{\infty}$ and for every $x\in U$ one has
$P_{i}=D^{i}f(x)$. For more details see for example \cite[Chapter~12]{C85}. 
\end{remark}

Let $T : E \rightarrow E$ be a bounded linear operator on a complex Banach space. Suppose that $T$ has the spectral gap property, by \cite{R55} there exists $\delta > 0$ such that if $\tilde{T} : E \rightarrow E$ is a bounded linear operator with $||T - \tilde{T}|| <\delta$ then $\tilde{T}$ has the spectral gap property. Moreover, $B(T , \delta) \ni \tilde{T} \mapsto \big(\rho(\tilde{T}) , P_{\rho(\tilde{T})}\big)$ is analytical, where $P_{\rho(\tilde{T})}$ is the spectral projection of $\tilde{T}$ with respect to the leading eigenvalue $\rho(\tilde{T})$. 

On the other hand, following the same proof of \cite[Proposition 4.1]{BCV16} we have

\begin{proposition}
Let $f : M \rightarrow M$ be a local homeomorphism on a compact and connected manifold and $E$ be a Banach algebra of functions such that $\mathcal{L}_{f, \phi}$ is a bounded linear operator of $E$ on $E$, for all $\phi \in E$. Then 
$$
E \ni \phi \rightarrow \mathcal{L}_{f,\phi}
$$
is analytical, where we endow on the counterdomain the topology generated by operator norm. In particular, $\R \ni t \mapsto \mathcal{L}_{f,t\phi}$ is real analytic.
\end{proposition}

Note that if $T = \mathcal{L}_{f,\phi}|_{E}$ has the spectral gap property we have that $P_{\rho(\tilde{T})}(g) = \int g d\nu_{\phi} \cdot h_{\phi}$. Define $SG(E) := \{\phi \in E : \mathcal{L}_{f,\phi}|_{E} \text{ has the spectral gap property} \}$, we then concluded:

\begin{corollary}\label{analy}
$SG(E) \subset E$ is an open subset and the following map is analytical:
$$
SG(E) \ni \phi \mapsto \big(\rho(\mathcal{L}_{f,\phi}) , h_{\phi} , \nu_{\phi}).
$$
\end{corollary} 

One weaker spectral property is called \textit{quasi-compactness}:

\begin{definition}
Given $E$ a complex Banach space and $T:E \to E$ a bounded linear operator, we say that $T$ is quasi-compact if there exists $0<\sigma<\rho(T)$ and a decomposition of $E=F\oplus H$ as follows: $F$ and $H$ are closed and $T$-invariant, $\dim F < \infty$, $\rho(T|_F)>\sigma$ and $\rho(T|_H) \leq \sigma$.
\end{definition}

In general, spectral gap implies quasi-compactness.
Later we will proof that, in the context of Theorem \ref{mainthC}, quasi-compactness is a sufficient condition for spectral gap. 

A definition equivalent to quasi-compactness can be given via the \textbf{essential spectral radius}:

\begin{definition}
Given $E$ a complex Banach space and $T:E \to E$ an bounded linear operator, define
$$\rho_{ess}(T):=\inf\{r>0; \;sp(L)\setminus \overline{B(0,r)} \text{ contains only eigenvalues of finite multiplicity}\}$$
\end{definition}

Thus quasi-compactness is equivalent to having $\rho_{ess}(T) < \rho(T)$, and so estimates on the essential spectral radius and the spectral radius will be of the utmost importance.




\section{Proof of the main results}\label{Proofs}

This sections is devoted to proving  Theorems \ref{mainthA}, \ref{mainthC} and Corollary \ref{mainthD}.

\subsection{Thermodynamical phase transition}

This section is devoted to the proof of Theorem \ref{mainthA}.

Given a unidimensional local diffeomorphism $f:\Sc^1 \to \Sc^1$ and an $f-$invariant measure $\mu$, we define the Lyapunov exponent for this measure $\chi_\mu(f):=\int \log|Df|d\mu$. Our goal is to understand the smoothness of the topological pressure function

$$P(t):=\sup\{h_\mu(f)-t\chi_\mu(f),\; \mu \text{ is an } f-\text{invariant probability} \},\; t \in \R,$$
 which is known, by the Variational Principle (see e.g. \cite{OV16}), that this supremum can be taken only over the space of $f-$ergodic probabilities $M_e(f)$. Thus, denote the following sets of measure with non-negative and non-positive Lyapunov exponents:

$$M_e^+(f):=\{\mu\in M_e(f),\;\chi_\mu(f)\geq 0 \} \text{ and }$$ 
$$M_e^-(f) :=\{\mu\in M_e(f),\; \chi_\mu(f)\leq 0\},$$
and the pressure functions restricted to each of these sets

 $$P_+(t):=\sup\{P_\mu(t),\; \mu \in M_e^+(f) \} \text{ and }$$
$$P_-(t):=\sup\{P_\mu(t),\; \mu \in M_e^-(f)\},$$ 
with $P_\mu(t):=h_\mu(f)-t\chi_\mu(f)$. With these notations, we prove the following lemmas

\begin{lemma}\label{L2}
Let $f:\Sc^1 \to \Sc^1$ be a $C^1$  local diffeomorphism such that $M^+_e(f)$ is non-empty, then the function  $P_+(t)$ is convex, non-increasing and has a zero in $(0,1]$.
\end{lemma}
\begin{proof}
$P_+$ is non-increasing. For each $\mu \in M_e^+(f)$ fixed, the line $P_\mu$ is non-increasing and thus for any $s>t$ we have $P_\mu(s)\leq P_\mu(t)$. Taking supremum over all $\mu$, it follows that $P_+(s)\leq P_+(t)$ and $P_+$ is non-increasing. $P_+$ is in fact convex. Since $P_+$ is supremum of straight lines $\R \ni t \mapsto P_\mu(t)$, where $\mu \in M_e^+(f)$, then $P_+$ is convex.
Finally, note that $$P_+(0)=\sup\{h_\mu(f), \mu \in M_e^+(f)\}=h_{top}(f)> 0$$ (see Remark \ref{remar21}) and  by Ruelle-Margulis's inequality (see \cite{Rue78}) $$P_+(1)=\sup\{h_\mu(f)-\chi_\mu(f), \mu \in M_e^+(f)\}\leq 0.$$ Thus, as $P_+$ is convex, in particular continuous, it has a zero in $(0,1]$. 
\end{proof}

The Figure \ref{Fig2} gives a geometric intuition of the function  $P_+$ where each line is the graph of $t \mapsto P_{\mu}(t)$, for a fixed ergodic probability $\mu$ in $M_e^+(f)$:

\begin{figure}[htb]
\includegraphics[width=8.1cm, height=4.5cm]{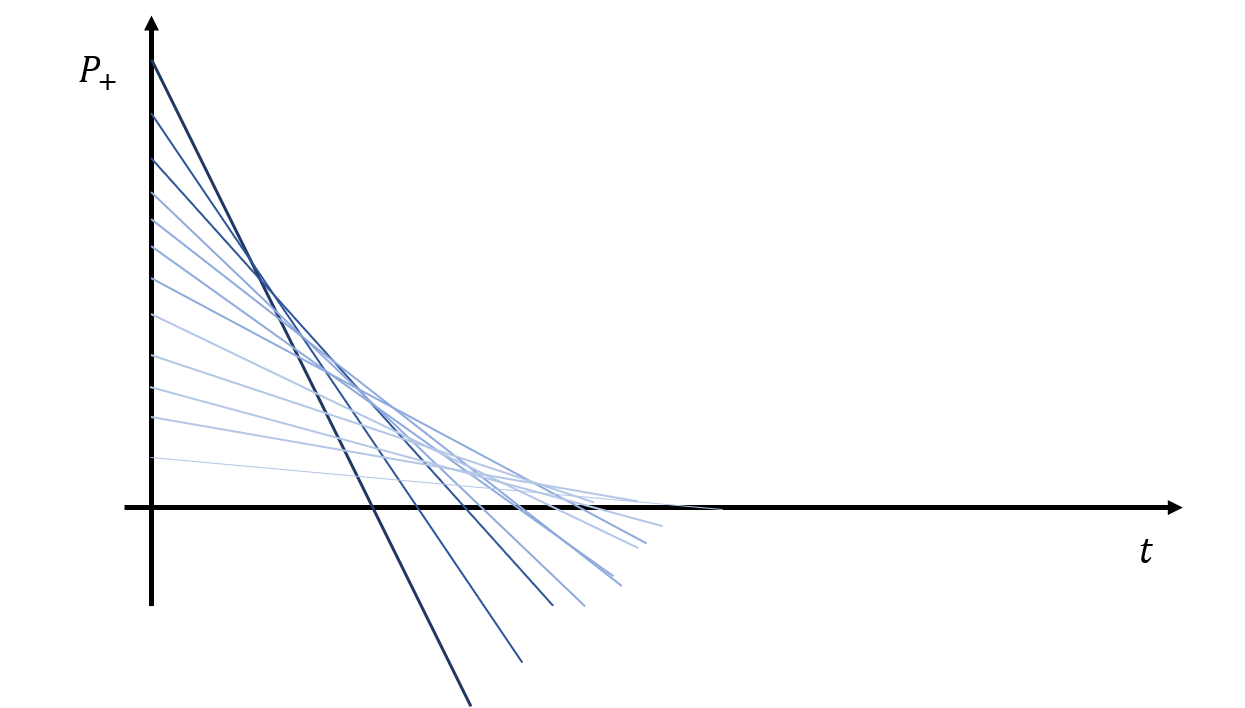}
\caption{}
\label{Fig2}
\centering
\end{figure}

\begin{lemma}\label{L3}
Let $f:\Sc^1 \to \Sc^1$ be $C^{1}-$local diffeomorphism  such that $M^-_e(f) \neq \emptyset$ then  $t \mapsto P_-(t)$ is linear and not decreasing.
\end{lemma}
\begin{proof}
 By Ruelle-Margulis's inequality (see \cite{Rue78}), if $\chi_\mu(f)\leq 0$ then $h_\mu(f)=0$ and $P_\mu(t) = -t \chi_\mu(f)$. 
  Then $P_\mu$ is linear and not decreasing, so $P_- = P_{\mu}$ where $\chi_{\mu}(f) = \min_{\nu \in \mathcal{M}_{1}(f)}\chi_{\nu}(f)$. By \cite[Theorem 3]{CLR03} 
 $\mu$ can be taken ergodic, and therefore
 $P_-=P_\mu$ is linear and not decreasing. 
 \end{proof}

The Figure \ref{Fig3} gives a geometric intuition of the function  $P_-$ where each line is the graph of $t \mapsto P_{\mu}(t)$, for a fixed ergodic probability $\mu$ in $M_e^-(f)$:

\begin{figure}[ht]
\includegraphics[width=7cm, height=4cm]{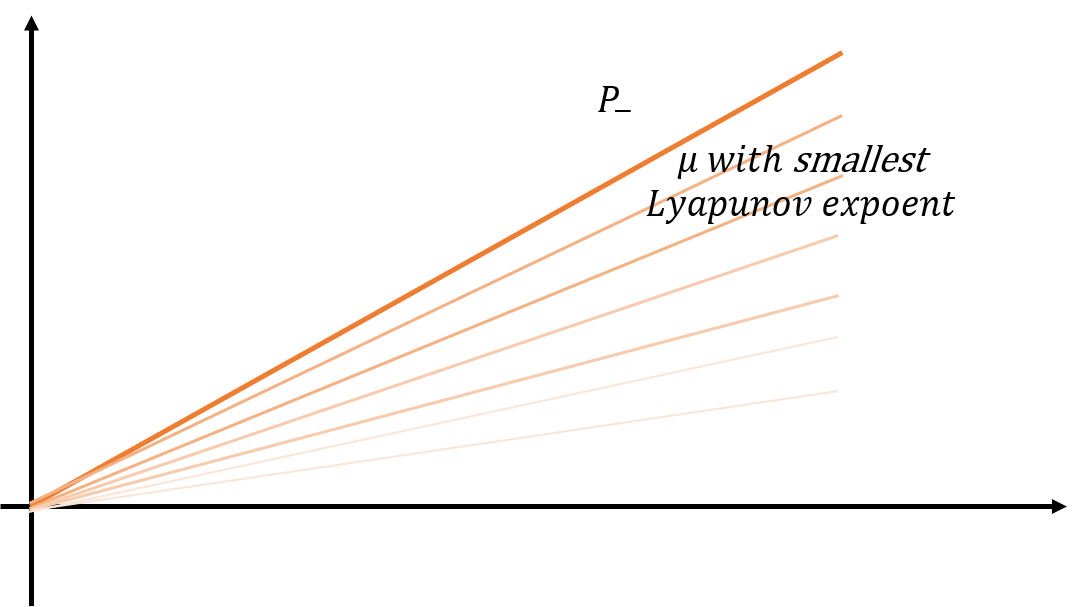}
\caption{}
\label{Fig3}
\centering
\end{figure}

It follows from definition of $P$ that $P=\max\{P_-,P_+\}$. Moreover, when $h_{top}(f) = 0$ we have $P_{\nu}(t) = -t\chi_{\nu}(f)$ and thus: $P(t) = -t\chi_{\mu_{1}}(f)$ for all $t > 0$, where $\chi_{\mu_{1}}(f) = \min_{\nu \in \mathcal{M}_{1}(f)}\chi_{\nu}(f)$, and $P(t) = -t\chi_{\mu_{2}}(f)$ for all $t < 0$, where $\chi_{\mu_{2}}(f) = \max_{\nu \in \mathcal{M}_{1}(f)}\chi_{\nu}(f)$.

\break
Using Lemmas \ref{L2} and \ref{L3}, we prove the main theorem.
\begin{proof}[Proof of Theorem \ref{mainthA}]

Since $h_{top}(f)>0$, by Ruelle-Margulis's inequality (see \cite{Rue78}) $f$ admits some ergodic invariant probability with positive Lyapunov exponent, hence $M_e^+(f)\neq \emptyset $. Using again the 
Ruelle-Margulis's inequality we have that $P_-(0) = 0 < h_{top}(f) = P_+(0) $. On the other hand, by \cite[Theorem 3]{CLR03} there exists $\mu \in \mathcal{M}_{e}(f)$ such that $\chi_{\mu}(f) = \min_{\nu \in \mathcal{M}_{1}(f)}\chi_{\nu}(f)$. Since $f$ is not an expanding dynamics, it follows from \cite[Theorem 5]{CLR03} that $\chi_{\mu}(f) \leq 0.$
 Thus, we will discuss the possibilities for this ergodic probability $\mu$:
 
 \quad
 
\begin{itemize}
\item{\textbf{Case $\chi_\mu(f)=0$}}. In this case, by previous Lemma $P_-=P_\mu\equiv 0$. Thus, fixed $t_0 \in (0,1]$ the lower zero of $P_+$, that is, $t_{0} := \min\{t \in \R : P_+(t) = 0\}$ we have:
$$
P(t)=
\begin{cases}
P_+(t), & \text{ for } t\leq t_0\\
0, & \text{ for }t \geq t_0
\end{cases}
$$
This means that $P$ cannot be real analytic in $t_0$ because all its lateral derivatives on the right are zero, but $P$ is not identically zero for $t < t_{0}$. 

\quad

\item{\textbf{Case} $\chi_\mu(f)<0$.} In this case, $P_-$ is not trivially zero, but an increasing line. Thus, since
$P_-(t) \xrightarrow[t\to+\infty]{}+\infty$ and $P_+(t)$ is not increasing there exists  $t > 0$ such that $P_+(t) \leq P_-(t)$. Define $t_{1} := \min\{ t \geq 0: P_+(t) \leq P_-(t)\}$. Since $P_+(0) > P_-(0) = 0$ then $t_{1} > 0$. Hence
$$
P(t)=
\begin{cases}
P_+(t), & \text{ for }t\leq t_1\\
P_-(t), &  \text{ for } t\geq t_1
\end{cases}
$$
Note that $P$ cannot even be differentiable in $t_{1}$, in fact suppose by absurd that $P$ is differentiable in $t_{1}$. Then
$$
0 \neq -\chi_{\mu}(f) = \lim_{t \searrow t_{1}}\frac{P(t) - P(t_{1})}{t - t_{1}} = \lim_{t \nearrow t_{1}}\frac{P(t) - P(t_{1})}{t - t_{1}} \leq 0,
$$
where the last inequality follows from $P_+$ being non-increasing. This is absurd.
\end{itemize}

We concluded in both cases that $f$ has phase transition with respect to $-\log |Df|$.
\end{proof}

The Figures \ref{Fig4} and \ref{Fig5} gives a geometric intuition of the pressure function in each case previous.

\begin{figure}[ht]
\includegraphics[width=8.1cm, height=4.5cm]{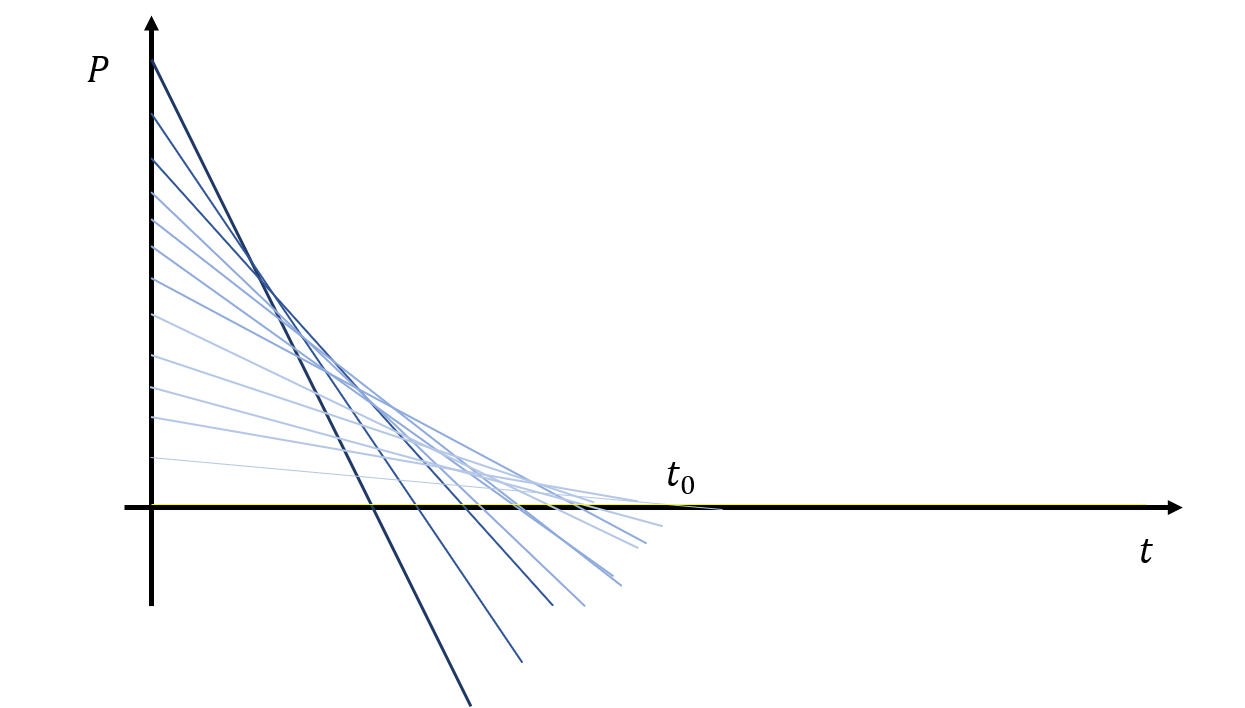}
\caption{Case $\chi_\mu(f)=0$}
\label{Fig4}
\centering
\label{caso0}
\end{figure}

\begin{figure}[ht]
\includegraphics[width=7cm, height=4cm]{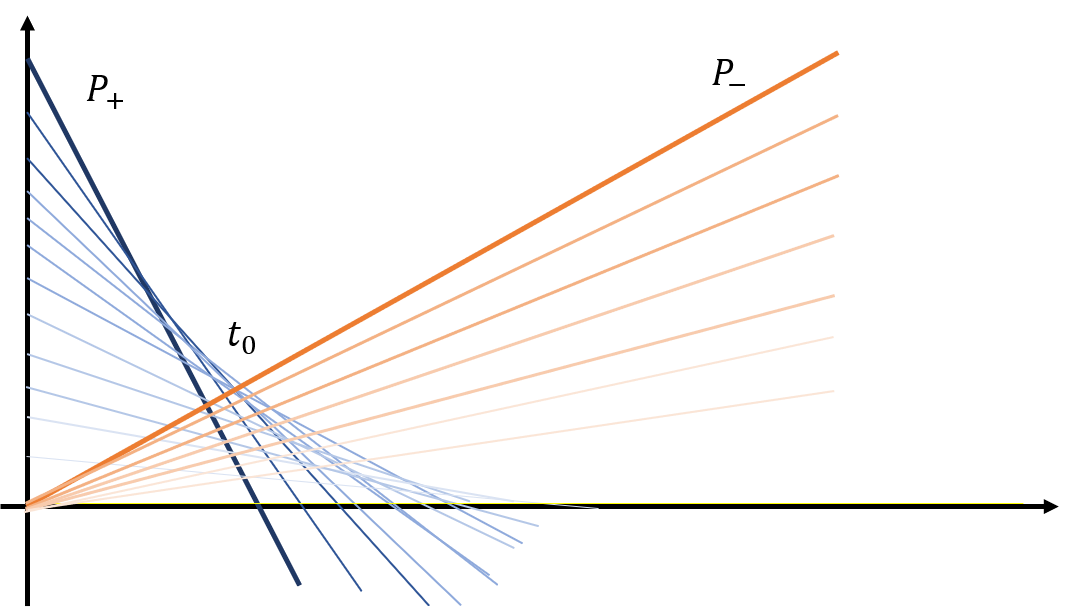}
\caption{Case $\chi_\mu(f)<0$}
\label{Fig5}
\centering
\label{caso-}
\end{figure}

\subsection{Effective phase transitions}

This section is devoted to the proof of the Theorem \ref{mainthC}. Suppose that $f$ is a transitive local diffeomorphism of the circle with degree at least two. First we remember that $f$ is conjugate to $dx$ mod $1$ where $d=\deg(f)$, by \cite[Corollary 4.3]{CM86}.

From this conjugacy, we then find that $f$ is expansive, topologically exact, and it admits generating partition by domains of injectivity. In particular, we can apply Rokhlin's formula (see \ref{rokf}).

Also, since $f$ is transitive and has  positive topological entropy, every ergodic measure has non-negative Lyapunov exponents,  by \cite[Corollary 1.2]{A20}\footnote{See Apendix \ref{Ap1} for more detalis}, and therefore $P=P_+$ and $P$ is non-increasing, by  Lemma \ref{L2}. 
Furthermore, $t_0:=\inf\{t\in(0,1]; P(t)=0\}$ is the parameter of the thermodynamical phase transition in Theorem \ref{mainthA} and $P(t) = 0$ for all $t \geq t_{0}$.

\

\subsubsection{Absence of  spectral gap}

This section is devoted to the proof of spectral phase transition. In fact, we will show:

\begin{proposition}\label{prop1}
Let $E = BV[\Sc^1], C^{\alpha}(\Sc^{1} , \C)$ or $C^{r}(\Sc^{1} , \C)$  and let $f:\Sc^1 \to \Sc^1$ be a transitive $C^{1}-$local diffeomorphism with $Df \in E$. 
If $f$ is not an expanding  dynamics then $\mathcal{L}_{f, -t_{0}\log|Df|}$ has not the spectral gap property acting on $E$. 
\end{proposition}

Now in order to prove Proposition \ref{prop1}, we must prove $P_{top}(f,-t\log|Df|)=\log \rho (\Lo_{f,-t\log|Df|}|_E)$.
The case $E=BV(\Sc^1)$ is a direct consequence of \cite[Theorem 3]{BK90}, since $f$ is continuous and monotone in each element of the generating partition.
For the case $E=C^{\alpha}(\Sc^1,\C)$ or $C^{r-1}(\Sc^1,\C)$ we followed the ideas from \cite[Proposition 3.5]{Ba00} and proved the following:

\begin{lemma}\label{Lemaxi}
Let $f:\Sc^1 \rightarrow \Sc^1$ be a transitive $C^r-$local diffeomorphism, let $E=BV[\Sc^{1}], C^{\alpha}(\Sc^1,\C)$ or $C^{r}(\Sc^1,\C)$, and let $\phi \in E$ be a real continuous function. If $\Lo_{f,\phi}\varphi=\lambda\varphi$,
with $|\lambda|=\rho(\Lo_{f,\phi}|_E)$ and $\varphi \in E\setminus \{0\}$, then
$\Lo_{f,\phi}|\varphi|=\rho(\Lo_{f,\phi}|_E)|\varphi|.$ Furthermore, $\rho(\Lo_{f , \phi}|_{C^0})=\rho(\Lo_{f,\phi}|_{E})$, $\varphi$ is bounded away from zero and $\dim\ker(\Lo_{f,\phi}|_E - \lambda I) = 1$.
\end{lemma}

\begin{proof}
Take $\xi \in \C$ such that $\lambda = \rho(\Lo_{f,\phi}|_E) \xi$ and $|\xi| = 1$.

By Mazur's separation theorem, there exists a probability $\nu$ with $$(\Lo_{\phi|C^{0}})^\ast\nu=\rho(\Lo_{\phi|C^{0}})\nu.$$ Therefore $\rho(\Lo_\phi|_{C^0})\leq \rho(\Lo_\phi|_{E})$.\\

\textit{Claim: supp $\nu=\Sc^1$, where supp $\nu$ means the support of the measure $\nu$.}\\

\noindent Indeed, let $A\subset \Sc^1$ be an open set. Then 

$$\nu(A)=\int \chi_A d\nu=\int \dfrac{\Lo_\phi^n(\chi_A)}{\rho(\Lo_\phi|_{C^0})^n}d\nu$$ for every $n\in \N$. By topological exactness, take $N>0$ such that $f^N(A)=\Sc^1$. Thus given $x \in \Sc^1$ we can find $y_0 \in f^{-N}(x)\cap A$, and $$\Lo_\phi^N\chi_A(x)=\sum_{f^Ny=x} e^{S_N\phi(y)} \chi_A(y) \geq e^{S_N\phi(y_0)}\chi_A(y_0)=e^{S_N\phi(y_0)}\geq e^{N\inf\phi}.$$ Finally
$$\nu(A)\geq \dfrac{e^{N\inf\phi}}{\rho(\Lo_\phi|_{C^0})^N}>0.$$ 
Which proves the claim.

\

Now define $$s(x):=\begin{cases}
\varphi(x)/|\varphi(x)|, \text{if } \varphi(x) \neq 0 \\
1,\;\; \text{if } \varphi(x)=0.
\end{cases}$$
 Note that $s|\varphi|=\varphi$, thus $$\Lo_{\phi}(s|\varphi|)=\Lo_{\phi}\varphi=\lambda \varphi = \lambda s|\varphi|.$$ This implies that

$$
    \Lo_{\phi}\left(\dfrac{s}{\lambda (s \circ f) }\cdot |\varphi|\right)(x) = \sum_{f(y)=x} e^{\phi(y)} \dfrac{s(y)}{\lambda s(x)}|\varphi|(y) = 
    \dfrac{1}{\lambda s(x)} \Lo_{\phi}(s|\varphi|)(x)= 
    $$
    $$
    \dfrac{1}{\lambda s(x)}\lambda s(x) |\varphi|(x) 
    \Rightarrow \Lo_{\phi}\left(\dfrac{s}{\lambda (s\circ f)}|\varphi|\right)=|\varphi|.
$$
Since $\nu$ is an eigenmeasure:
\begin{align*}
    \int \dfrac{s}{\rho(\Lo_{\phi}|_E) \xi \cdot s \circ f }|\varphi| d\nu \cdot \rho(\Lo|_{C^0})= \int |\varphi| d\nu.
\end{align*}
Hence 
    \begin{align*}
    \dfrac{\rho(\Lo_\phi|_{C^0})}{\rho(\Lo_\phi|_{E})}\int \dfrac{s}{\xi\cdot s \circ f }|\varphi| d\nu  = \int |\varphi| d\nu.
\end{align*}
Since $\left|\dfrac{s}{\xi \cdot s\circ f}\right|\equiv 1$ and $|\varphi| \geq 0$, we must have $\rho(\Lo_\phi|_{C^0})=\rho(\Lo_\phi|_{E})$ and
$s(x)=\xi \cdot s\circ f(x)$, for $|\varphi|$d$\nu$-a.e. $x$. Then
$s(x)|\varphi|(x) = \xi \cdot s\circ f(x) |\varphi|(x)$ for $\nu$-a.e. $x$, that is, $\varphi= \xi \cdot s\circ f \cdot |\varphi|$, since $\text{supp }\nu=\Sc^1$. Now, applying the transfer operator, we have:
\begin{align*}
    \Lo_\phi(|\varphi|)=\Lo_\phi\left(\dfrac{\varphi}{\xi\cdot s\circ f}\right)
    =\dfrac{\Lo_\phi(\varphi)}{\xi\cdot s}= \rho(\Lo_\phi|_{E}) \dfrac{\varphi}{s}=\rho(\Lo_\phi|_{E})|\varphi|.\\
\end{align*}

Next we prove that $\varphi$ is bounded away from zero.
Indeed, let $x \in \Sc^1$  be such that $|\varphi|(x)=0$. Since
$$\sum_{f(y)=x}e^{\phi(y)}|\varphi|(y)=\rho(\Lo_{\phi|E})|\varphi|(x)=0,$$
we must have $\varphi(y)=0$ for each pre-image $y$ of $x$, and arguing by induction all pre-orbits $\{f^{-n}(x) : n \in \N\}$ of $x$ must be a zero of $|\varphi|$. Since the set of pre-orbits of $x$ is dense by topological exactness, we would have $|\varphi|\equiv 0$, which is a contradiction. 
Since $\varphi \in E$, and $\varphi$ is bounded away from zero, $|\varphi|\in E$.

Finally we prove that the eigenspace associated with $\lambda$ is unidimensional.
Let $\varphi_1, \varphi_2$ be eigenfunctions of $\lambda$. Note that there exists $t \in \C$ and $x_0\in \Sc^1$ such that $\varphi_1(x_0)+t\varphi_2(x_0)=0$. Since
$\Lo_{f,\phi}(\varphi_1+t \varphi_2) = \lambda(\varphi_1+t \varphi_2)$, and
by the boundness from zero of non-zero eigenfunctions, we conclude that $\varphi_1+t \varphi_2 \equiv 0$.
\end{proof}

\begin{corollary}\label{medida}
Let $f:\Sc^1 \rightarrow \Sc^1$ be a transitive $C^r-$local diffeomorphism, let $E= BV[\Sc^{1}], C^{\alpha}(\Sc^1,\C)$ or $C^{r}(\Sc^1,\C)$, and let $\phi \in E$ be a continuous real function. If $\mathcal{L}_{f, \phi|E}$ has the spectral gap property then there exists a unique probability $\nu_{\phi}$ on $\Sc^{1}$ such that $ (\mathcal{L}_{f, \phi|E})^{\ast}\nu_{\phi} = \rho(\mathcal{L}_{f, \phi|E})\nu_{\phi}$.
\end{corollary}
\begin{proof}
Note that $\mathcal{L}_{\phi}$ is a continuous operator on $C^{0}(\Sc^{1} , \C)$. Thus, the probability $\nu$ used in the proof of the previous lemma exists and $(\mathcal{L}_{\phi|E})^{\ast}\nu = \rho(\mathcal{L}_{\phi|C^{0}})\nu$. The previous lemma assures us that $\rho(\mathcal{L}_{\phi|C^{0}}) = \rho(\mathcal{L}_{ \phi|E})$.

Finally let us show the uniqueness. Let $\eta$ be probability on $\Sc^{1}$ such that $(\mathcal{L}_{ \phi|E})^{\ast}\eta = \rho(\mathcal{L}_{ \phi|E})\eta$. The spectral gap property implies that given $\varphi \in E$ then for some scalar $z$: $$\lim_{n \mapsto \infty}\frac{\Lo^{n}_{\phi}\varphi}{\rho(\mathcal{L}_{ \phi|E})^{n}} = z \cdot \lim_{n \mapsto \infty}\frac{\Lo^{n}_{\phi}1}{\rho(\mathcal{L}_{ \phi|E})^{n}}\Rightarrow \eta(\varphi) = z = \nu(\varphi) \Rightarrow \eta = \nu.$$
\end{proof}

\begin{remark}
Given $\Lo_{f,\phi|E}$ with the spectral gap property, by previous lemma there exists an unique $h_{\phi} \in E$ such that $\int h_{\phi} d\nu_{\phi} = 1$. Moreover, denote the $f-$invariant probability $h_{\phi}d\nu_{\phi}$ by $\mu_{\phi}$. Note that supp $\nu_{\phi} = \Sc^{1}$ and $h_{\phi} > 0$, by previous lemma, thus supp $\mu_{\phi} = \Sc^{1}.$
\end{remark}

For the next lemma we follow the proof for expanding maps whose main part is Rohklin's formula (see eg \cite{OV16}). For convenience of the reader we will present the highlights of the proof.

\begin{lemma}\label{Lemapress}
Let $f:\Sc^1 \rightarrow \Sc^1$ be a transitive $C^r-$local diffeomorphism, let $E=C^{\alpha}(\Sc^1,\C)$ or $C^{r}(\Sc^1,\C)$  and let $\phi \in E$ be a real function. If $\Lo_{f,\phi|E}$ has the spectral gap property, then $P_{top}(f,\phi)=\log \rho(\Lo_{f,\phi}|_E)$ and the probability $\mu_\phi$ is an equilibrium state for $f$ with respect to $\phi$.
\end{lemma}

\begin{proof}
From Lemma \ref{Lemaxi} the transfer operator $\Lo_\phi$ has  bounded away from 0 eigenfunction $h_\phi \in E$ for the eigenvalue $\lambda_\phi=\rho(\Lo_{\phi}|_E)$ :
                   $$\Lo_\phi(h_\phi)=\lambda_\phi h_\phi,$$
    and the dual operator $\Lo^*_\phi$ of $\Lo_\phi$ has an eigenmeasure $\nu_\phi$ for the eigenvalue $\lambda_\phi$, in particular
    $$\int \Lo_\phi(g)d\nu_\phi = \lambda_\phi \int g d\nu_\phi \; \forall g \in E.$$

Furthermore $\mu_\phi=h_\phi\nu_\phi$ defines an invariant probability.
Note that 
$$\mathcal{L}_{\phi +\log h_{\phi} - \log (h_{\phi} \circ f)}^{\ast}\mu_{\phi} = \lambda_{\phi}\mu_{\phi},$$ by  \cite[Corollary 4.2.7]{PU10} it follows that $f$ admits jacobians relatively to $\mu_\phi$ given by 
$$J_{\mu_\phi}f=\lambda_{\phi} \frac{e^{-\phi}h_{\phi}\circ f}{h_{\phi}}.$$

See that $f$ admits generating partition, then Rokhlin's formula (see Theorem \ref{rokf}) holds and we get:
$$h_{\mu_\phi}(f)=\int \log J_{\mu_{\phi}}f d\mu_\phi=\log\lambda_\phi - \int \phi d\mu_\phi+\int(\log h_{\phi}\circ f-\log h_{\phi}) d\mu_\phi.$$
Since $\mu_\phi$ is invariant and $\log h_\phi$ is bounded ($h_\phi$ is bounded away from $0$), the last term is zero and thus 
$$
h_{\mu_\phi}(f) + \int \phi d\mu_\phi = \log\lambda_\phi. 
$$

Now we prove that $\mu_\phi$ is an equilibrium state. Let $\eta$ be an $f-$invariant probability satisfying
$$h_\eta(f)+\int \phi d\eta \geq \log\lambda_\phi.$$
Let $g_\eta:=\frac{1}{J_\eta f}$ be,  and let $g_{\mu_{\phi}}:=\frac{1}{J_{\mu_{\phi}} f}$ be. Note that

$$ \sum_{y \in f^{-1}(x)}g_{\mu_{\phi}}(y)=\dfrac{1}{\lambda_{\phi} h_{\phi}(x)} \sum_{y \in f^{-1}(x)}e^{\phi(y)}h(y)=\dfrac{\Lo_{f,\phi} h_{\phi}(x)}{\lambda_{\phi} h_{\phi}(x)}=1$$ 

for all $x \in \Sc^{1}$. Also, since $\eta$ is invariant by $f$ 
$$
\sum_{y \in f^{-1}(x)} g_\eta (y)=1 \text{ for } \eta-\text{a.e.} x.
$$
By Rohklin's formula,
$$0 \leq h_\eta(f)+\int \phi d\eta - \text{log}\lambda_{\phi} = \int(-\log g_\eta + \phi - \log \lambda_{\phi}) d\eta. $$
Since $\eta$ is invariant and by definition of $g_{\mu_{\phi}}$, the integral above equals

$$\int(-\log g_\eta + \log g_{\mu_{\phi}} + \log h_{\phi}\circ f - \log h_{\phi}) d\eta=\int \log \dfrac{g_{\mu_{\phi}}}{g_\eta} d\eta $$

Then, by definition of $g_\eta$

$$\int\log\dfrac{g_{\mu_{\phi}}}{g_\eta}d\eta=\int(\sum_{y \in f^{-1}(x)}g_\eta(y) \log \dfrac{g_{\mu_{\phi}}}{g_\eta}(y))d\eta(x)$$

Since $t \mapsto \log t$ is concave, we have:

$$
    \sum_{y \in f^{-1}(x)} g_\eta(y) \log\dfrac{g_{\mu_{\phi}}}{g_\eta}(y) \leq \log \sum_{y \in f^{-1}(x)} g_\eta(y) \dfrac{g_{\mu_{\phi}}}{g_\eta}(y) 
    = \log \sum_{y \in f^{-1}(x)}g_{\mu_{\phi}}(y)=0
$$

for $\eta$-a.e. $x$. Finally we get 
$$ h_\eta(f) + \int \phi d\eta - \log\lambda_{\phi} = \int \dfrac{g_{\mu_{\phi}}}{g_\eta} d_\eta = 0.$$

Thereby $\mu_\phi$ is an equilibrium state for $f$ with respect to $\phi$ and from the Variational Principle $$P(f,\phi)=\log \rho(\Lo_\phi).$$
\end{proof}

\begin{proof}[Proof of the Proposition \ref{prop1}]
If $\Lo_{f,-t_{0}\log|Df|}|_E$ had the spectral gap property on $(0,1]$, then by Lemma \ref{Lemapress} and Corollary \ref{analy}, the topological pressure function 
$$(-\varepsilon + t_{0}, t_{0} +\epsilon) \ni t \mapsto P_{top}(f,-t\log|Df|) = \log(\rho(\Lo_{f,-t\log|Df|}|_E))$$ would be analytical, contradicting  Theorem \ref{mainthA}.
\end{proof}

\

\subsubsection{Lack of spectral gap after transition}

\

\begin{proposition}\label{propat}
Let $E = BV[\Sc^{1}], C^{\alpha}(\Sc^{1} , \C)$ or $C^{r}(\Sc^{1} , \C)$  and let $f:\Sc^1 \to \Sc^1$ be a transitive $C^{1}-$local diffeomorphism not invertible, with $Df \in E$. Suppose also that $f$ is not an expanding dynamics. If $\Lo_{f,-s\log |Df|}$ has the spectral gap property on $E$, for some $s \in \R$, then the topological pressure function $t \mapsto P(t)$ is strictly convex in a neighborhood of $s$. Moreover,
$\Lo_{f , -t\log |Df|}|_{E}$ has not spectral gap property for all $t \geq  t_0$.
\end{proposition}
\begin{proof}
Suppose that $\Lo_{f,-s\log |Df|}$ has the spectral gap property on $E$, for some $s \in \R$. Then again, by Corollary \ref{analy}, spectral gap holds for any small perturbation of this operator and thus, since the operator varies analytically with the potential, for any small perturbation of the potential. Let $\lambda_{\phi}$ be denoting $\rho(\Lo_{f,\phi|E})$ and fix $\psi \in E$. Then the function $T(t):=\dfrac{\lambda_{-s\log |Df| +i t \psi}}{\lambda_{-s\log |Df|}}$ is well defined for $t \in \R$ in a small neighborhood of zero and is analytical. By Nagaev's method we have $\sigma^2_{f,-s\log |Df|}(\psi)=-D_t^2 T(t)|_{t=0}$  where $\sigma^2:=\sigma^2_{f,-s\log |Df|}(\psi)$ is the variance of the \textit{Central Limit Theorem} with respect to dynamics $f$, probability $\mu_{-s\log |Df|}$ and observable $\psi$. Moreover $\sigma = 0$ if, only if,  there exists $c \in \R$ and $u \in E$ such that 
$$\psi(x)= c + u\circ f(x) - u(x) , \text{ for } \mu_{-s\log |Df|}-\text{a.e. } x,$$
(see e.g. \cite[Lecture 4]{S12} for more details). Note that
$$
\frac{d \, \lambda_{-s\log |Df| +i t \psi}}{dt} = i\frac{d \, \lambda_{\phi}}{d\phi}|_{\phi = -s\log |Df| +i t \psi} \cdot \psi \Rightarrow 
$$
$$
\frac{d^2 T(t)}{dt^{2}}|_{t=0} = - \frac{ \frac{d^{2}\,\lambda_{\phi}}{d^{2} \phi}|_{\phi = -s\log |Df| } \cdot (\psi ,  \psi) }{\lambda_{-s\log |Df|} } = - \frac{ \frac{d^{2}\,\lambda_{-s\log |Df|  + t \psi}}{d^{2} t}|_{t = 0} }{\lambda_{-s\log |Df|} }.
$$
In particular, we wil take $\psi = -\log |Df| + \int \log |Df| d\mu_{-s\log|Df|}$.

 Define $G(t) := P_{top}( f , -s\log|Df| + t\psi)$. Note that 
 $$G(t) = P(t + s) + t\int \log |Df| d\mu_{-s\log|Df|},$$
  for all $t \in \R$. Furthermore, for $t \in \R$ close enough to zero we have that $\mathcal{L}_{f , -s\log|Df| + t\psi}$ has the spectral gap property and thus $G(t) = \log \lambda_{-s\log|Df| + t\psi}$, by Lemma \ref{Lemapress}. Hence:
 
$$
\frac{d^{2}\,\lambda_{-s\log |Df|  + t \psi}}{d^{2} t } = e^{G(t)}\Big(  \big(G'(t)\big)^{2} + G''(t) \Big) \Rightarrow
$$
$$
\sigma^{2} = \sigma^2_{f,-s\log |Df|}(\psi)= \frac{ \frac{d^{2}\,\lambda_{-s\log |Df|  + t \psi}}{d^{2} t}|_{t = 0} }{\lambda_{-s\log |Df|} } = \big(G'(0)\big)^{2} + G''(0).
$$
It follows from \cite{W92} that $G'(0) = \int \psi d\mu_{-s\log |Df|} = 0$, we conclude then that $\sigma^{2} = G''(0)$.

Suppose by contradiction that $G''(0) = 0$. Hence, by Nagaev's method, there exists $c \in \R$ and $u \in E$ such that 

$$\log|Df(x)|= c + u\circ f(x) - u(x) , \text{ for } \mu_{-s\log |Df|}-\text{a.e. } x,$$
where $\mu_{-s\log |Df|}$ is the equilibrium state obtained via the spectral gap property, which we know has full support. Therefore $\log|Df|\equiv c + u\circ f - u$, and thus the Birkhoff time average  $\dfrac{1}{n} S_n \log|Df|$ converges uniformly to $c$. Consequently, $c$ is the unique Lyapunov exponent of $f$. Since $h_{top}(f) > 0$ then $c >0$, by Ruelle-Margulis inequality (see \cite{Rue78}). Thus $f$ would be an expanding dynamics, by \cite{CLR03}. This contradicts the hypotheses.

Therefore $P''(s) = G''(0) > 0$, which implies that $P$ is strictly convex in a neighborhood of $s$.

Lastly, by Theorem \ref{mainthA} the topological pressure function $t \to P(t)$ is constant in $[t_{0} , +\infty)$. Then $\Lo_{f , -t\log |Df|}|_{E}$ has not spectral gap property for all $t \geq  t_0$.
\end{proof}

\

\subsubsection{Spectral gap before of transition}

First we prove, for transitive local diffeomorphism in the circle, that if the transfer operator is \text{quasi-compact} then it has the spectral gap property.

\begin{proposition}\label{LemaEss}
Let $E = BV[\Sc^{1}], C^{\alpha}(\Sc^{1} , \C)$ or $C^{r}(\Sc^{1} , \C)$, let $f:\Sc^1 \to \Sc^1$ be a transitive $C^{r}-$local diffeomorphism and let  $\phi \in E$ be a real function. If $\Lo_{f,\phi}|_E$ is quasi-compact, then it  has the spectral gap property.
\end{proposition}
\begin{proof}

By quasi-compactness of  $\mathcal{L}_{\phi}$ then it admits a peripheral eigenvalue $\xi \rho(\mathcal{L}_{\phi}|_E)$, with $|\xi| = 1$. It follows from Lemma \ref{Lemaxi} that
$\lambda=\rho(\mathcal{L}_{\phi})$ is a simple eigenvalue. It remains to show that it is the only peripheral eigenvalue. 

Following the proof of the Lemma \ref{Lemaxi}, let $\varphi\in E$ be an eigenvector, not null, for $\xi \rho(\mathcal{L}_{\phi}|_E)$,
define $$s(x):=\begin{cases}
\varphi(x)/|\varphi(x)|, \text{if } \varphi(x) \neq 0 \\
1, \text{otherwise},
\end{cases}$$

\

\textit{Claim:  $\mathcal{L}_{\phi }(s^n|\varphi|)=\rho(\mathcal{L}_{\phi}|_E)\xi^n s^n |\varphi|$, for all $n \in \Z$.}

\

\noindent \textit{Case $E=BV[\Sc^1]$:} We know that there exists an eigenmeasure $\nu \in \ker\left( (\Lo_{\phi}|C^{0})^{\ast} - \lambda I \right)$ which has full support, since $f$ is topologically exact (see claim in Lemma \ref{Lemaxi}). Applying \cite[Proposition 3.5]{Ba00} the claim holds.

\

\noindent \textit{Case $E=C^\alpha(\Sc^1, \C)$ or $C^{r}(\Sc^1,\C):$} Following the proof of Lemma \ref{Lemaxi} we  have 
$s(x)=\xi \cdot s\circ f(x)$, for $|\varphi|$d$\nu$-a.e. $x$. Thus $s^n=\xi^n \cdot s^n\circ f$, $|\varphi|$d$\nu$-a.e., which means 
$s^n |\varphi| = \xi^n \cdot s^n\circ f |\varphi|$, $\nu$-a.e.. Since $\text{supp }\nu=\Sc^1$ we have that the functions coincide a.e.. Now applying the transfer operator we have:
$$
    \mathcal{L}_{\phi}(s^n |\varphi|)=\mathcal{L}_{\phi}(s^{n-1}s|\varphi|)=\mathcal{L}_{\phi}(\xi^{n-1}s^{n-1}\circ f \cdot s|\varphi|)=
    $$
    $$
    \xi^{n-1}s^{n-1}\mathcal{L}_{\phi}(s|\varphi|)=\xi^{n-1}s^{n-1}\rho(\mathcal{L}_{\phi})\xi\cdot s|\varphi|=
    \xi^n \rho(\mathcal{L}_{\phi}) s^n |\varphi|. 
$$
and the claim holds.\\

By Lemma \ref{Lemaxi} $\varphi$ is bounded away from zero, thus $s $ and $ |\varphi| \in E$. Thereby $s^n|\varphi| \in E,$  and, by previous claim, it follows that
$\xi^n \cdot \rho(\Lo_{\phi}|_E) \in sp(\Lo_{\phi})$ for all $n\in \Z.$

Since the set $\{\xi^n: n \in \Z\}$ forms a subgroup of the circle then it is either dense or periodic. However by quasi-compactness of the transfer operator all eigenvalues are isolated and it can't be a dense subgroup, so there is a $k>0$ such that 

$$ \xi^k \rho(\Lo_{\phi}|_E)=\rho(\Lo_{\phi}|_E).$$

We already know $|\varphi|$ is an eigenvector for the spectral radius, thus we have both 
\begin{align*}
    \Lo_{\phi}^k(|\varphi|)=\rho(\Lo_{\phi})^k |\varphi|\\
    \Lo_{\phi}^k(\varphi)=\rho(\Lo_{\phi})^k \varphi
\end{align*}
and then
$$ \Lo_{\phi}^k(|\varphi|-\varphi)=\rho(\Lo_{\phi})^k (|\varphi|-\varphi).$$ 
We can assume without loss of generality that there exists $x_0$ such that $\varphi(x_0)>0$, otherwise it is enough replacing $\varphi$ by $\frac{\varphi}{\varphi(x_{0})}$. Thus  $(|\varphi|-\varphi)(x_{0})=0$. Similarly to the end of the proof of  Lemma \ref{Lemaxi} we conclude that $|\varphi|-\varphi$ is null in $\{y \in f^{-kn}(x_{0}) : n  \in \N\}$.  Topological exactness implies that $\{y \in f^{-kn}(x_{0}) : n \in \N\}$ is dense. Thus continuity implies that $|\varphi|-\varphi \equiv 0$. Hence
 $$\xi \rho(\Lo_{\phi|E})\varphi=\Lo_{\phi}\varphi = \Lo_{\phi}|\varphi| = \rho(\Lo_{\phi|E}) \varphi.$$  Finally $\xi=1$ and $\rho(\Lo_{\phi|E})$ is the unique peripheral eigenvalue. 
\end{proof}

\

\begin{proof}[Proof of the Theorem \ref{mainthC}] By Proposition \ref{propat}, to prove the Theorem is enough proving that $\Lo_{f,-t\log|Df|}|_{E}$ has the spectral gap property for all $t<t_0$. The proof will be divided in two cases.

 \

\textit{Case $E=BV[\Sc^1]:$}

\

 Let $R_t:=\lim_{n\to\infty} \|\Lo^{n}_{f,-t\log|Df|}1\|_{\infty}^{1/n} \text{ and } \Tilde{R}_t:=\lim_{n\to\infty}\| \frac{1}{ |Df^n|^{t}}\|_\infty^{1/n}$ be. Then, by  \cite[Theorem 3.2]{Ba00} and  \cite[Theorem 3]{BK90},
$R_t=\rho(\Lo_{f,-t\log|Df|}|_{BV}) = e^{P(t)}$ and, by  \cite[Theorem 1]{BK90}, $\Tilde{R}_t \geq \rho_{ess}(\Lo_{f,-t\log|Df|}|_{BV})$. 

\begin{lemma}
$t \geq 0 \Rightarrow \Tilde{R}_t\leq 1$.
\end{lemma}
\begin{proof}
Obviously $\tilde{R}_{0} = 1$. Take $t > 0$. Suppose by absurd that $\Tilde{R}_t > 1$. Then there exists $\lambda > 1$  and $n_{0} \in \N$ such that $\left|\left|  \frac{1}{ |Df^n|^{t}}\right|\right|_\infty^{1/n} \geq \lambda,$ for all $n \geq n_{0}$. Thus, for each $n \geq n_{0}$ there exists $x_{n} \in \Sc^{1}$ with:
 $$
 \frac{1}{ |Df^n(x_{n})|^{\frac{t}{n}}} \geq \lambda \Rightarrow |Df^{n}(x_{n})| \leq \lambda^{\frac{-n}{t}}.
$$
Define $\mu_{n} := \frac{1}{n}\sum_{j= 0}^{n-1}\delta_{f^{j}(x_{n})}$. Since the space of probabilities on $\Sc^{1}$ is compact, in the weak-$^{\ast}$ topology, take $\mu := \lim_{k \to \infty}\mu_{n_{k}}$. Note that $\mu$ is $f-$invariant, furthermore
$$
\int\log |Df| d\mu = \lim_{k \to \infty}\frac{1}{n_{k}}\sum_{j= 0}^{n_{k}-1}\log |Df(f^{j}(x_{n_{k}}))| = \lim_{k \to \infty}\frac{1}{n_{k}}\log |Df^{n_{k}}(x_{n_{k}})| \leq 
$$
$$
-\frac{1}{t}\log \lambda < 0.
$$
By ergodic decomposition theorem we can write $\mu = \int \mu_{x}d\mu(x)$, where each $\mu_{x}$ is $f-$ergodic (see e.g. \cite{OV16}). Hence
$$
\int\log |Df| d\mu = \int(\int\log |Df| d\mu_{x})d\mu(x) = \int \chi_{\mu_{x}}(f)d\mu(x) \leq -\frac{1}{t}\log \lambda < 0 \Rightarrow
$$
$$
 \exists \,\mu_{x} \text{ com } \chi_{\mu_{x}}(f) < 0.
$$
This is absurd because $\mathcal{M}^{-}_{e}(f) = \emptyset$. 
\end{proof}

\

Note that $R_t=e^{P(t)}>1$, for all $t \in [0,t_0)$.  Applying the previous lemma, it follows that $\Lo_{f,-t\log|Df|}|_{BV}$ is quasi-compact for all $t\in[0,t_0)$ and by Lemma \ref{LemaEss} it has the spectral gap property for $t \in [0,t_0)$.

\

\textit{Case $E=C^{\alpha}(\Sc^1,\C)$ or $C^{r}(\Sc^1,\C):$}

\begin{lemma}\label{decre}
The pressure function $t \mapsto P(t)$ is strictly decreasing in $(\infty,t_0]$.
\end{lemma}
\begin{proof}
Since $f$ is expansive, classic Ergodic Theory implies the entropy function  $\mu \mapsto h_\mu(f)$ is upper semi-continuous (see e.g. \cite{OV16}) and therefore $f$ has at least one ergodic equilibrium state $\mu_\phi$ for a given potential $\phi$.
For all $t<t_0$ we have $P(t)>0$, thus $P(t)=h_t -t \lambda_t$, where $h_t:= h_{\mu_{-t\log |Df|}}$ and $\lambda_t:=\int \log|Df| d\mu_{-t\log |Df|}>0$  by Ruelle-Margulis's inequality. Take $t<s<t_0$, then by supremum definition
$$
     P(t)  = h_t - t \lambda_t \geq h_s - t \lambda_s  
   = h_s - s \lambda_s + (s-t)\lambda_s 
   >  h_s - s \lambda_s = P(s)
$$
and $P$ is strictly decreasing.
\end{proof}

\begin{remark}\label{decre2}
Note that, the previous estimates also shows  that if $t < s$ and $\lambda_{s} > 0$ then $P(t) > P(s).$
\end{remark}

For the following lemmas we need to estimate the essential spectral radius, which in the given spaces will require different approaches. First for $E=C^\alpha(\Sc^1,\C)$,  the results \cite[Theorems 1 and 2]{BJL96} reads, in our context, as the following

\begin{theorem}\label{CH}
Let $f:\Sc^1 \rightarrow \Sc^1$ be a $C^1-$local diffeomorphism and let $\phi \in C^\alpha(\Sc^1, \R)$ be then 
$$
\rho_{ess}(\Lo_{f,\phi}|_{C^\alpha})=\lim_{n \to \infty}\|\Lo^n_{f,\phi-\alpha \log|Df|}1\|_\infty^{1/n}
\text{ and }
$$
$$
\rho(\Lo_{f,\phi}|_{C^\alpha})= \max\{\lim_{n \to \infty}\|\Lo^n_{f,\phi}1\|_\infty^{1/n} \,,\,  \rho_{ess}(\Lo_{f,\phi}|_{C^\alpha})\}.$$
\end{theorem}

\begin{remark}
Note that $\lim_{n \to \infty}\|\Lo^n_{f,g}1\|_\infty^{1/n}=\rho(\Lo_{f,g}|_{C^0}).$
\end{remark}

For $E=C^{r}(\Sc^1,\C)$,  the results \cite[Theorems 1 and 2]{CL97} summarize to:

\begin{theorem}\label{Lat}
Let $f:\Sc^1 \rightarrow \Sc^1$ be a $C^r-$local diffeomorphism and let $\phi \in C^{r}(\Sc^1, \R)$ be then
$$ \rho_{ess}(\Lo_{f,\phi}|_{C^k}) \leq \exp\big[\sup_{\mu \in \mathcal{M}_{1}(f)}\{h_\mu(f)+\int\phi \text{d}\mu-k\chi_\mu(f)\}\big] \text{ and }$$
$$ \rho(\Lo_{f,\phi}|_{C^k}) \leq \exp\big[\sup_{\mu \in \mathcal{M}_{1}(f)}\{h_\mu(f)+\int\phi \text{d}\mu\}\big],$$
for $k=0,1,\dots, r.$
\end{theorem}

Note that in this case $$ \rho_{ess}(\Lo_{f,-t\log|Df|}|_{C^k}) \leq e^{P(k+t)} \text{ and } \rho(\Lo_{f,-t\log|Df|}|_{C^k}) \leq e^{P(t)}.$$

Then on continuous H\"older functions,
$$\rho_{ess}(\Lo_{f,-t\log|Df|}|_{C^\alpha}) = \rho(\Lo_{f,-(t+\alpha)\log|Df|}|_{C^0}) \leq e^{P(\alpha +t)}$$
by Theorems \ref{CH} and \ref{Lat}.

In either case we have 
\begin{lemma}\label{essest}
 $\rho_{ess}(\Lo_{f,-t\log|Df|}|_{E})\leq e^{P(k+t)},$ for some $k>0.$
\end{lemma}

We start proving spectral gap for non-positive $t$.
\begin{lemma}
$\Lo_{f,-t\log|Df|}|_{E}$ has spectral gap for all $t\leq 0.$
\end{lemma}
\begin{proof}

\textit{Claim: $\Lo_{f,0}|_E$ has spectral gap.}

\

\noindent In fact, since $P$ is strictly decreasing on $[0,t_0)$
$$\rho_{ess}(\Lo_{f,0}|_E)\leq e^{P(k)}<e^{P(0)}=e^{h_{top}(f)}=\deg(f)=\rho(\Lo_{f,0}|_E).$$
Hence $\Lo_{f,0}|_E$ is quasi-compact and by Lemma \ref{LemaEss} it has spectral gap.

\

Let $\tau_1:=\inf\{t<0; \Lo_{f,-t\log|Df|}|_E \text{ has the spectral gap property}\}$ be. Suppose by contradiction that $\tau_1 > -\infty.$ Since the spectral gap property is open (see Corollary \ref{analy}) then $\Lo_{f,-\tau_{1}\log|Df|}|_E$ has not spectral gap property.

\

\textit{Claim: $\rho(\Lo_{f,-\tau_1 \log|Df|}|_E)=e^{P(\tau_1)}$.}

\

\noindent In fact, for all $t \in (\tau_1,0]$ the transfer operator $\Lo_{-t\log|Df|}|_E$ has the spectral gap property, then $\rho(\Lo_{-t\log|Df|}|_{E})=e^{P(t)}$ and $(\tau_1,0]\ni t \mapsto \rho(\Lo_{-t\log|Df|}|_{E})$ is decreasing, by Lemma \ref{Lemapress} and Lemma \ref{decre}. 
Now suppose by contradiction there is a $t_n \searrow \tau_1$ such that $\rho(\Lo_{-\tau_1\log|Df|}|_{E}) < \rho(\Lo_{-t_n\log|Df|}|_{E})-\varepsilon$. By semi-continuity of spectral components (see e.g. \cite{K95}) we have sp$(\Lo_{-t_n\log|Df|}|_{E})\subset B\big(0,\, \rho(\Lo_{-\tau_1\log|Df|}|_{E})+\epsilon/2\big)$.
Therefore $$\rho(\Lo_{-\tau_1\log|Df|}|_{E}) \geq \sup_{t\in(\tau_1,0]}\rho(\Lo_{-t\log|Df|}|_{E})=\sup_{t\in(\tau_1,0]}e^{P(t)}=e^{P(\tau_1)}.$$

\noindent On the other hand, by Theorem \ref{Lat} $\rho(\Lo_{-\tau_1\log|Df|}|_{C^{r}})\leq e^{P(\tau_1)}$. Furthermore, by Theorem \ref{CH} and, again,  Theorem \ref{Lat}
$$
\rho(\Lo_{-\tau_1\log|Df|}|_{C^{\alpha}}) = \max\{\rho(\mathcal{L}_{-\tau_1\log|Df|} |_{C^{0}}) \,,\,  \rho_{ess}(\Lo_{-\tau_1\log|Df|}|_{C^\alpha})\} \leq
$$
$$
 \max\{ e^{P(\tau_{1})} , e^{P(\tau_{1} + \alpha)} \}.
$$
Since $t \mapsto P(t)$ is decreasing we obtain that $\rho(\Lo_{-\tau_1\log|Df|}|_{C^{\alpha}}) \leq  e^{P(\tau_{1})} $. Thus we conclude $\rho(\Lo_{-\tau_1 \log|Df|}|_E)=e^{P(\tau_1)}$.

\

Therefore, 
$$\rho(\Lo_{-\tau_1\log|Df|}|_{E})=e^{P(\tau_1)}>e^{P(\tau_1+k)}\geq \rho_{ess}(\Lo_{-\tau_1\log|Df|}|_{E}),$$
and by Lemma \ref{LemaEss} $\Lo_{-\tau_1\log|Df|}|_{E}$ has the spectral gap property, contradicting the minimality of $\tau_1.$
\end{proof}

Let $\tau_2:=\sup\{t>0; \Lo_{f,-t\log|Df|}|_{E}$ has the spectral gap property $\}$, note that $\tau_2$ exists and is at most $t_0$.
Analogously to the second claim in the previous Lemma, it holds $\rho(\Lo_{f,-\tau_2\log|Df|}|_{E})=e^{P(\tau_2)}$. Since the spectral gap property is open, $\Lo_{f,-\tau_2\log|Df|}|_{E}$ has not spectral gap and thus $\rho(\Lo_{f,-\tau_2\log|Df|}|_{E})=\rho_{ess}(\Lo_{f,-\tau_2\log|Df|}|_{E})$ by Lemma \ref{LemaEss}.
\begin{lemma}
$\tau_2=t_0$
\end{lemma}
\begin{proof}
By the essential radius estimates on Lemma \ref{essest}, we get:
$$e^{P(\tau_2)}=\rho(\Lo_{-\tau_2\log|Df|}|_{E})=\rho_{ess}(\Lo_{-\tau_2\log|Df|}|_{E})\leq e^{P(\tau_2+k)}.$$

Since $P$ is decreasing we have $P(\tau_2)=P(\tau_2+k)$. Therefore $\tau_2\geq t_0$,  since $P(t)$ is strictly decreasing before $t_0$. Consequently, $\tau_2=t_0.$ 
\end{proof}
From definition of $\tau_2$ we get:
\begin{corollary}
$\Lo_{f,-t\log|Df|}|_{E}$ has spectral gap for all $t<t_0$.
\end{corollary}

This completes the proof of Theorem \ref{mainthC}.
\end{proof}

\

\subsubsection{Effective thermodynamical phase transition}

\begin{proof}[Proof of the Corollary \ref{mainthD}]
Just like Proposition \ref{prop1}, this is a direct consequence of $\rho(\Lo_{f,-t\log|Df|}|_{E})=e^{P(t)}$ for $t<t_0$ and Corollary \ref{analy}.
\end{proof}

\

\section{Examples, discussions and more questions}\label{disque}

\subsection{Examples}
In this section we present some explicit examples of local diffeomorphisms with different types of phase transition. First, given a $C^1 $ diffeomorphism $g: [0,1/2] \to [0,1]$ consider the circle map with identification $0=1$ defined as,

\begin{equation}\label{difeo}
 f(x)=\begin{cases}
 g(x), \text{ if } 0\leq x \leq 1/2 \\
 1-g(1-x), \text{ if } 1/2 < x \leq 1.
 \end{cases}
\end{equation}

\noindent The map $f$ defined this way is a $C^1$ local diffeomorphism.

\begin{example} \normalfont Let $a \in (0,2]$, then $g_a(x)=ax+(4-2a)x^2$ gives a continuous family of local diffeomorphisms $f_a(x)$ by equation \ref{difeo}. Note that $f_a'(x)=a+(8-4a)x$ and $f_a'(0)=a$, so that the system undergoes a sort of pitchfork bifurcation:

  \begin{figure}[h]
        \centering
        \includegraphics[scale=5]{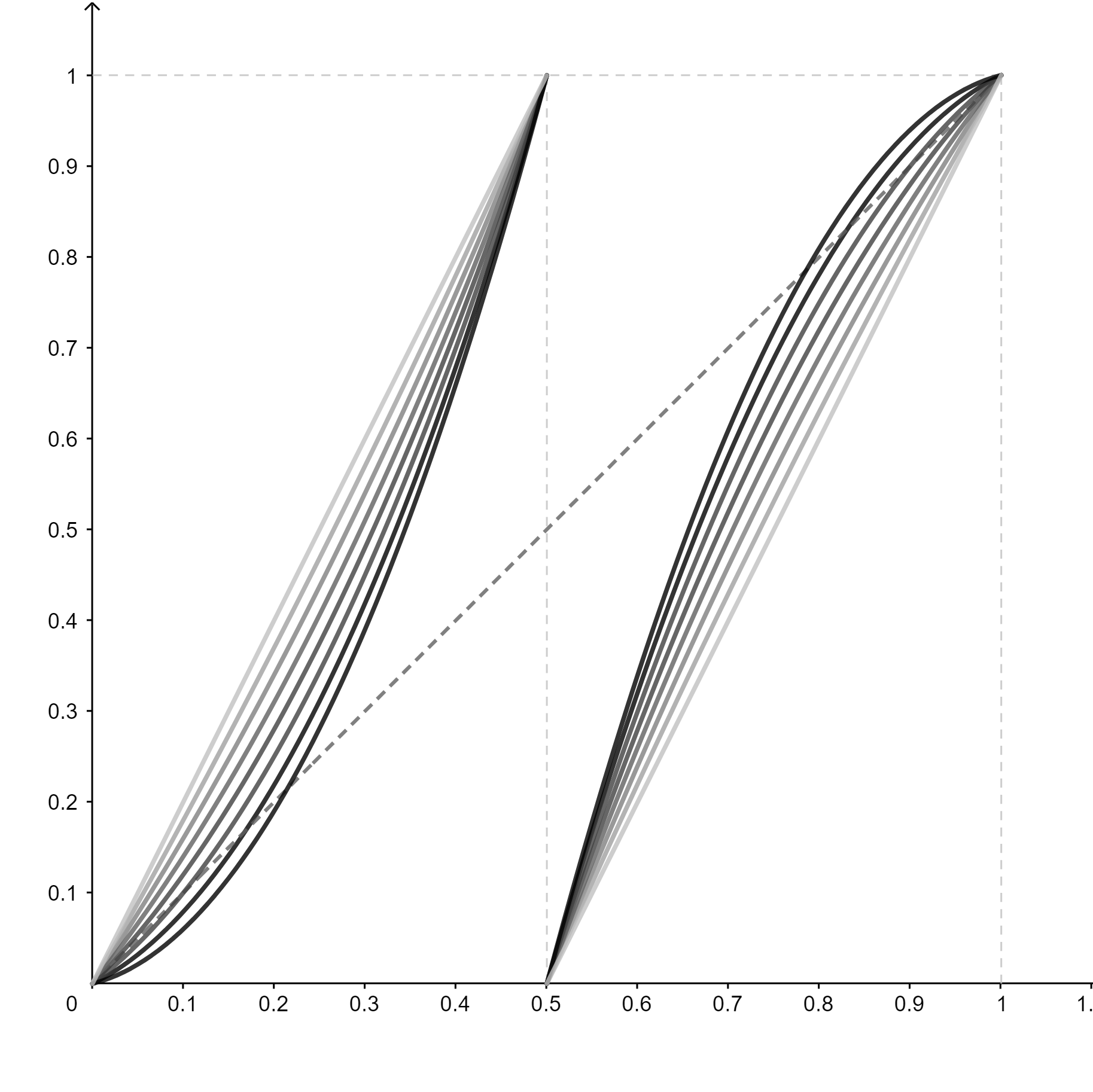}
        \caption{Collection of maps $f_a$}
        \label{bifurcation}
    \end{figure}

\begin{enumerate}
   
    \item For $a \in (1,2]$, the fixed point $0$ is expanding, and the map $f_a$ is in fact uniformly expanding. Note that for $a=2$, the map $f_a$ is the doubling map $2x \text{ mod } 1$.
   \begin{figure}[h]
        \centering
        \includegraphics[scale=5]{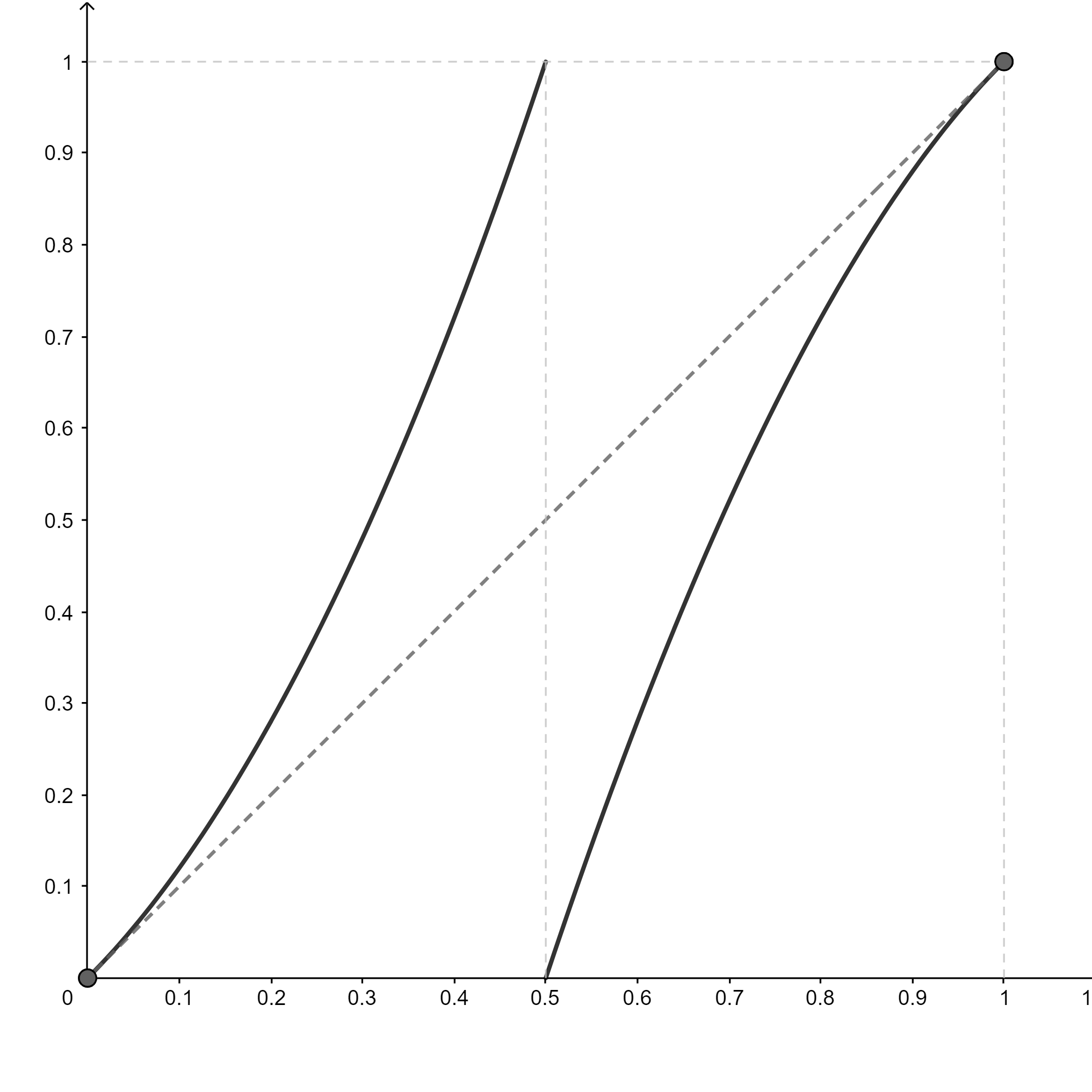}
        \caption{Case $a=1.5$,\; $g_{a}(x)=1.5x+x^2$}
        \label{Case1.5}
    \end{figure}
  
    Consequently the geometric pressure for these maps is strictly convex and decreasing and therefore has no phase transition.
    
        \item For $a=1$, the fixed point $0$ is indifferent, and the map $f_a$ obtained from $g_a(x)=x+2x^2$ is transitive and intermittent, that is, expanding everywhere except for $0$. Note that this is a circle version of the Manneville-Pomeau map on the circle.
    \begin{figure}[h]
        \centering
        \includegraphics[scale=5]{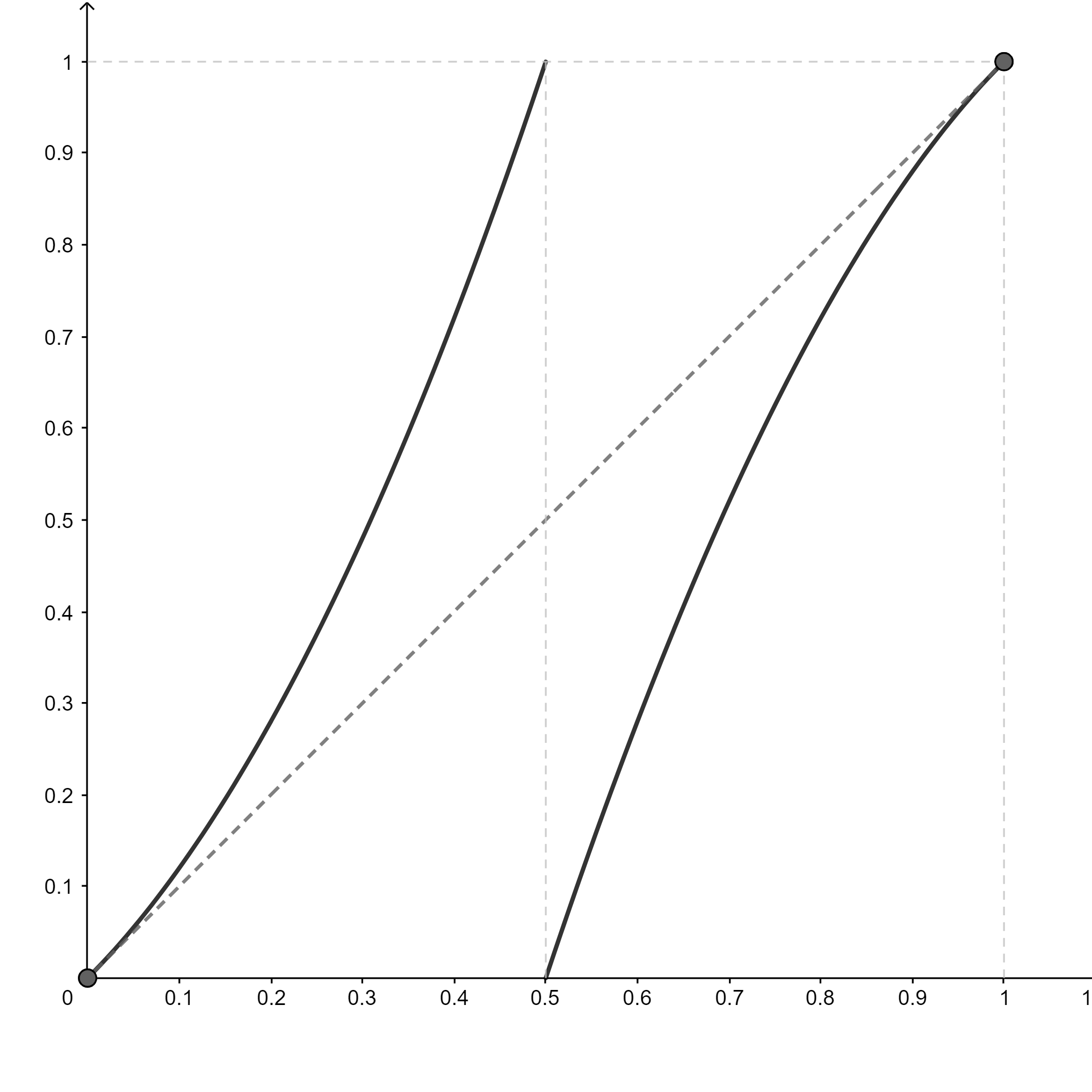}
        \caption{Case $a=1$,\; $g_{a}(x)=x+2x^2$}
        \label{Case1}
    \end{figure}
    
    Consequently, the Lyapunov exponent associated to the Dirac measure $\delta_0$ is zero and therefore the geometric pressure for this map fits into the case of Figure \ref{caso0} of Theorem \ref{mainthA}. Since this map is of class $C^{1+Lip}$, in particular  $\log|Df|$ is H\"older continuous, we have that Theorem \ref{mainthC} and Corollary \ref{mainthD} do apply for this map.
    
     \item For $a \in (0,1)$, the fixed point $0$ becomes attracting, and two other expanding fixed points emerge. In this case, the map $f_a$ is not transitive, where the small arc containing these three fixed points is invariant.

      
 Consequently the Lyaunov exponent associated to the Dirac measure $\delta_0$ is negative and therefore the pressure for these maps fits into the case of Figure \ref{caso-} of Theorem \ref{mainthA}. However  Theorem \ref{mainthC} and Corollary \ref{mainthD} do not apply to this case.
\end{enumerate}
\end{example}

\begin{example} \normalfont
As mentioned before, the main models for transitive non-expanding local diffeomorphisms on the circle are the Manneville-Pomeau maps obtained from $g_\alpha(x)=x+2^\alpha x^{\alpha+1}$ and equation \ref{difeo}. This maps are of class $C^{1+\alpha}$ but not of class $C^2$. In order to have both the fixed point $0$ and the middle point $1/2$ as roots of the second derivative, and thus proper inflexion points, we need polynomials of higher degree. So a $C^2$ version of the Manneville-Pomeau map is the map $f$ obtained from $g(x)=x+8x^3-8x^4$ in equation \ref{difeo}:
\begin{figure}[h!]
    \centering
    \includegraphics[scale=5]{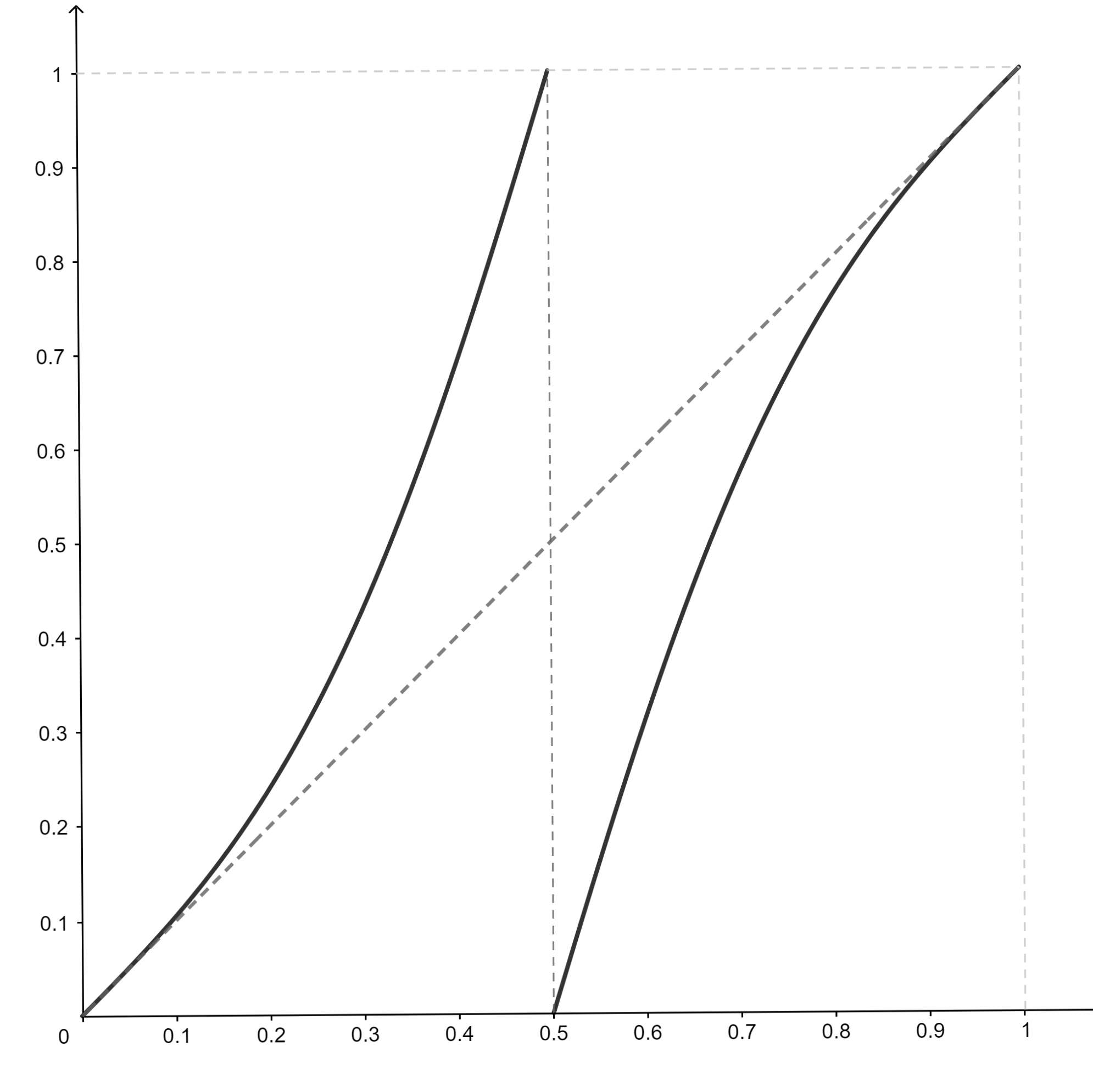}
    \caption{Intermittent map of class $C^2$}
    \label{ManPO.c2}
\end{figure}

We could again consider two families containing this map. The first is for a given $0<k\leq 2$ the family of functions $f_k(x)=kx+(16-8k)x^3+(8k-16)x^4$ that when applied to equation \ref{difeo}, gives a family of $C^2$ local diffeomorphisms on the circle. For $k<1$ the maps are non-transitive map with attracting fixed point, $k=1$ the $C^2$ Manneville-Pomeau map, and for $k>1$ uniformly expanding maps that converge to the Bernoulli map: $2x \mod 1$. Thus we have again a bifurcation, that is completely analogous to Figure \ref{bifurcation}.

Another interesting family of maps that have all the point $x=0$ as an  indifferent fixed point would be the following: consider $p \in [0,1]$ and define $b=b(p):=\left((\frac{1}{2})^{3+p}-\dfrac{4+p}{4+2p}(\frac{1}{2})^{2+p}\right)^{-1}$ and $a=a(p):=\dfrac{-b(4+p)}{4+2p}$ then $f(x)=x+ax^{3+p}+bx^{4+p}$ applied to equation \ref{difeo} gives a family of $C^2$ maps with indifferent fixed points, that is analogous to the family obtained from $g_\alpha(x)=x+2^\alpha x^{\alpha+1}$.
\end{example}

\

\subsection{Calculating the phase transition}\label{t0}

Let $f : \Sc^{1} \rightarrow \Sc^{1}$ be as in Corollary \ref{mainthD}. We obtain the transition parameter 
$$
t_{0} = \inf\{ t \in \R : P_{top}(f , -t\log |Df|) = 0\} = 
$$
$$
\sup\{t>0; \Lo_{f,-t\log|Df|}|_{E} \text{ has the spectral gap property }\} = 
$$
$$
\inf\{t>0; \Lo_{f,-t\log|Df|}|_{E} \text{ has not the spectral gap property }\},
$$ 
moreover $t_{0} \leq 1$.

We are interested in this section to discuss when we can make the transition parameter $t_{0}$ explicit, as well as obtain a characterization of $t_{0}$ in terms of the measures with positive Lyapunov exponent.

Initially we'll see cases where $t_{0} = 1$. The first one is when we have an \emph{a.c.i.p.} with positive Lyapunov exponent.

\begin{proposition}
Let $f : \Sc^{1} \rightarrow \Sc^{1}$ be as in Corollary \ref{mainthD}. Suppose that $f$ admits an $f-$invariant probability $\mu$ such that $\mu$ is absolutely continuous with respect to Lebesgue measure and has  positive Lyapunov exponent. Then $t_{0} = 1$. 
\end{proposition}
\begin{proof}
Since $\mu$ is absolutely continuous with respect to Lebesgue measure, by Pesin's entropy formula (see \cite{P77}), we have that $\mu$ is an equilibrium state with respect to geometric potential $- \log |Df|$. By hypothesis, $\mu$ has positive Lyapunov exponent. Thus, by remark \ref{decre2}, we have $0 = P(1) < P(t)$ for all $t < 1$. It implies that $t_{0} = 1.$
\end{proof}


The second case is associated with the case in which the loss of expansion is only caused by indifferent fixed point. Initially we need to introduce the notion of dynamical dimension. Fix $f : [0 ,1] \rightarrow [0 , 1]$ a piecewise monotonic map. Given $\mu \in \mathcal{M}_{e}(f)$ define $HD(\mu) := \inf\{HD(Y) : \mu(Y) = 1\}$, where $HD(Y)$ is the Hausdorff dimension of $Y$. We define $dD(f) := \sup\{HD(\mu) : h_{\mu}(f) > 0\}$, the dynamical dimension of $f$.

\begin{remark} Another way to think of the dynamical dimension of a map, in some contexts, is to consider the supremum of $HD(\Lambda)$ for all $\Lambda$  repellers for $f$. Take for instance a continuous map $f: X \to X$ of a compact subspace of the closed complex plane $X \subset \bar{\C}$ that can be analyticly extended to a neighbourhood $U(f)$ of $X$. Then \cite[Theorem 10.6.1]{PU10} guarantees that given any measure with positive Lyapunov exponent we obtain measures supported on repellers that approximate both entropy and Lyapunov exponents of $\mu$, thus by \cite[Corollary 4]{HZ12} we obtain an approximation also for the dimension of these measures.
\end{remark}

 In the paper \cite{H93} the aim is to find conditions such that $dD(f_{|A}) \geq z_{A} \geq HD(A)$, where $A$ is a compact $f-$invariant subset, $f_{|A}$ is transitive and $z_{A}  := \inf\{t >0 : P_{top}(f_{|A} , -t\log|Df_{|A}|) \leq 0\}$. In fact, if $|Df| > 1$ except on a finite set of  indifferent fixed points then $z_{A} \geq HD(A)$. In our context $A = \Sc^{1}$ and $z_{A} = t_{0}$, thus applying \cite[Theorem 6]{H93}:

\begin{proposition}
Let $f : \Sc^{1} \rightarrow \Sc^{1}$ be as in the Corollary \ref{mainthD}. Suppose that $|Df| > 1$ except on a finite set of fixed points on which $Df$ is equal to $1$. Then $t_{0} = 1$
\end{proposition}

In spite of these two cases, in general $t_ {0}$ does not need to be $1$ because, as we will see, $t_{0}$ is related to the support of measures with zero Lyapunov exponent.

\begin{proposition}
Let $f : \Sc^{1} \rightarrow \Sc^{1}$ be as in the Corollary \ref{mainthD}. Then $t_{0} = dD(f)$.
\end{proposition}
\begin{proof}
On the one hand, applying   \cite[Corollary 4]{HZ12} in our context we have that: if $\mu$ is an $f-$invariant ergodic Borel probability measure with $\chi_{\mu}(f) > 0$ then $HD(\mu) = \frac{h_{\mu}(f)}{\chi_{\mu}(f)}$. In particular, $P_{\mu}(HD(\mu)) = 0$. Thus, take $0 < t < t_{0}$ and $\mu_{t}$ such that $P_{\mu_{t}}(t) = P(t) > 0$. Hence, $0 < h_{\mu_{t}}(f) \leq\chi_{\mu_{t}}(f) $ and then
$$
P_{\mu_{t}}(HD(\mu_{t})) = 0 \Rightarrow HD(\mu_{t}) > t \Rightarrow dD(f) > t, \, \forall t < t_{0} \Rightarrow dD(f) \geq t_{0}.
$$
On the other hand, suppose by absurd that $dD(f) > t_{0}$. Then there exists $\mu \in \mathcal{M}_{e}(f)$ with $h_{\mu}(f)  >0$ and $HD(\mu) > t_{0}$. Therefore
$$
\frac{h_{\mu}(f)}{\chi_{\mu}(f)} > t_{0} \Rightarrow P_{\mu}(t_{0}) > 0 \Rightarrow P(t_{0})  > 0,
$$
this is absurd.
\end{proof}

The previous result is analogous to the generalised Bowen's formula obtained in \cite{PRL19} in the context of multimodal maps.

This discussion leads us to the following question:

\begin{mainquestion}
There exists a dense subset of the space of transitive, non invertible and non uniformly expanding local diffeomorphism on the circle whose thermodynamical phase transition occurs in the parameter $t =1$ ?
\end{mainquestion}

\

\subsection{Non-transitive case}\label{nontra}

In Theorem \ref{mainthC} and Corollary \ref{mainthD} it was assumed the transitivity hypothesis, the purpose of this section is to present a classic result on piecewise monotonic maps in order to indicate some perspectives on this case.

Given $f : [0 , 1] \rightarrow [0 , 1]$ a piecewise continuous and monotonic map, the papers \cite{H86}, \cite{HR89} and \cite{W88} guarantee us that a decomposition theorem on the nonwandering set $\Omega(f)$ holds:
$$
\Omega(f) = \bigcup_{j =1}^{\infty}L_{j} \cup L_{\infty} \cup W \cup P
$$
where the sets $L_{j}, L_{\infty}$ and $P$ are closed and $f-$invariant, and the intersection of two different sets in this union is finite or empty. Moreover,
\begin{itemize}
\item 
$f_{|L_{j}}$ is topologically conjugated to a 
transitive subshift of finite type.
\item $L_{\infty}$ is the disjoint union of finitely many compact, $f-$invariant sets $N_{i}$, such that $N_{i}$ is a Cantor-like and $f_{|N_{i}}$ is  minimal. In fact, there are only finitely many ergodic invariant Borel probability measures on each $N_{i}$ and $h_{top}(f_{|L_{\infty}}) = 0$.
\item $W$ consists of nonperiodic points, which are isolated in $\Omega(f)$.
\item $P$ consists of periodic points contained in nontrivial intervals $K \subset \Sc^{1}$ with the property that $f^{n}$ maps $K$ monotonically into $K$ for some n.
\item Given $\mu \in \mathcal{M}_{e}(f)$ there exists $\eta \in \mathcal{M}_{e}(f_{|L_{j}})$ such that $h_{\mu}(f) \leq h_{\eta}(f)$.
\end{itemize}

In particular, taking $f : \Sc^{1} \rightarrow \Sc^{1}$ a local diffeomorphism, define 
$$
P_{j}(t) := P_{top}(f_{|L_{j}} , -t\log |Df|) \,\;,\,\; P_{\infty}(t) := P_{top}(f_{|L_{\infty}} , -t\log |Df|)
$$
$$
\text{and } \, P_{p}(t) := P_{top}(f_{|P} , -t\log |Df|).
$$
 Note that $P_{\infty} \equiv 0$ and $P_{p}$ is linear in each of the intervals $(-\infty , 0]$ and $[0 , +\infty)$. 
In general, we expect that $f$ can have infinite thermodynamical phase transitions with respect to $-\log|Df|$.

Suppose now that  $f$ is not invertible and non transitive, $Df$ is H\"older continuous and that there is a unique $L_{j}$. Then
 $$
 P_{top}(f , -t\log|Df|) = \max\{P_{j}(t), P_{p}(t)\}
 $$ and $h_{top}(f_{|L_{j}}) = h_{top}(f) = \log deg(f)$. Hence, $f_{|L_{j}}$ is topologically conjugated to a full shift. We cannot apply directly Theorem \ref{mainthD} on $f_{|L_{j}}$ because $L_{j}$ will not be the circle or a union of intervals (in fact, typically $L_{j}$ will be a Cantor set). Despite this fact, we can follow the proof of  Theorem \ref{mainthD} and the fundamental point for its conclusion are the estimates about the spectral radius and essential spectral radius of the transfer operator $\mathcal{L}_{f_{|L_{j}}, -t\log|Df|}|_{C^{\alpha}}$.
 Thereby, if it is possible to overcome such obstacle we would conclude that $f$ can have thermodynamical phase transitions with respect to $-\log|Df|$ in at most in three points.

\

\subsection{A little bit about the higher dimensional context}\label{hdc}

In this section we want to discuss a little about the Problem \ref{probA} and Conjecture \ref{conjA} presented in the introduction. 

We start with an example. Given the expanding dynamics $f : \Sc^{1} \xrightarrow[x \mapsto 2x \text{ mod } 1]{} \Sc^{1} $ and the irrational rotation $R_{\alpha} : \Sc^{1} \xrightarrow[y \mapsto y + \alpha]{} \Sc^{1}$, take the skew product 
$$
F : \Sc^{1} \times \Sc^{1} \xrightarrow[(x, y) \mapsto \big(f(x) , R_{\alpha}(y) \big)]{} \Sc^{1} \times \Sc^{1}.
$$
Given the H\"older potential $\phi : \Sc^{1} \times \Sc^{1} \rightarrow \R$ define the H\"older continuous potential 
$$\tilde{\phi} : \Sc^{1}  \xrightarrow[x \mapsto \int \phi(x , y) d\Leb(y) ]{} \R.$$   Note that
$$
P_{top}(F , \phi) = \sup_{\mu \in \mathcal{M}_{1}(f)}\{h_{\mu}(f) + h_{Leb}(R_{\alpha}) + \int \phi(x , y)dLeb(y)d\mu(x) \,\}= P_{top}(f , \tilde{\phi}).
$$
Since $f$ is an expanding dynamics then $C^{\alpha}(\Sc^{1}  , \R) \ni g \mapsto P_{top}(f , g)$ is analytical. We conclude that $C^{\alpha}(\Sc^{1}  \times \Sc^{1} , \R) \ni \phi \mapsto P_{top}(F , \phi)$ is analytical, even though $F$ is not expanding or hyperbolic dynamics.

This example shows in high dimension we cannot expect to obtain thermodynamical phase transition for all local diffeomorphism that are not expanding or hyperbolic. Note that the previous example is not topologically conjugated to an expanding or hyperbolic dynamics. Moreover there is no mixing between the central direction and the expanding direction, that is a rigid condition. Thus, we propose the following questions:

\begin{mainquestion}
Let $f : M \rightarrow M$ be a local diffeomorphism such that $f$ is topologically conjugated to an expanding or hyperbolic dynamics, but $f$ is not an expanding or hyperbolic dynamics. Then $f$ has thermodynamical phase transition with respect to some H\"older continuous potential ?
\end{mainquestion}

\begin{mainquestion}
Generically a local diffeomorphism is expanding, hyperbolic or has thermodynamical phase transition with respect to some H\"older continuous potential ?
\end{mainquestion}

Remember that the paper \cite{Kl20} assures us that for any map of the circle which is expanding outside an arbitrarily flat neutral point, the set of H\"older potentials such that $\mathcal{L}_{f,\phi|C^{\alpha}}$ has the spectral gap property is dense, in the uniform topology. Thus it leads us to the following question:

\begin{mainquestion}
Let $f : M\rightarrow M$ be  $C^{1}-$local diffeomorphism on a manifold $M$, with $h_{top}(f) > 0$. The set of potentials $\phi$, in a suitable Banach space (H\"older continuous or smooth functions),  such that  $\mathcal{L}_{f, \phi}$ has not the spectral gap property acting on a suitable Banach space (H\"older continuous or smooth functions) can be dense or residual in the uniform topology ?
\end{mainquestion}


\vspace{.3cm}
\subsection*{Acknowledgements.}
This work is part of the second author's PhD thesis at Federal University of Bahia. TB was partially supported by
CNPQ (Grants PQ-2021) and CNPq/MCTI/FNDCT project 406750/2021-1, Brazil. VC was supported by CAPES-Brazil. The authors are deeply
grateful to Armando Castro, Paulo Varandas and Juan Rivera-Letelier for useful comments.


\begin{appendix}

\section{Attracting periodic orbits}\label{Ap1}

In this appendix we first transcribe part of the text from V. Ara\'ujo  \cite{A20}, concerning attracting periodic orbits, then we apply it to our context.

\

In what follows $M$ is a connected compact finite $d$-dimensional manifold $M$, with $d \geq 2$. Let $f: M \rightarrow M$ be a $C^{1}$ map such that $\inf _{x \in M}\|D f(x)\|>0$. The Subadditive Ergodic Theorem  ensures that the largest asymptotic growth rate

$$
\chi(x)=\lim _{n \rightarrow+\infty} \ln \left\|D f^{n}(x)\right\|^{1 / n}
$$
exists for all $x$ on a total probability subset since $\ln ^{+}\|D f\|=\max \{0, \ln \|D f\|\}$ is $\mu$-integrable for each $f$-invariant probability measure $\mu$.

In what follows we write $A_{k}^{-}(x)=\lim \inf _{n \rightarrow+\infty} n^{-1} \sum_{j=0}^{n-1} \ln \left\|D f^{k}\left(f^{k j} x\right)\right\|$. We recall that $p \in M$ belongs to a periodic orbit (with period $\tau$) if there exists $\tau \in \mathbb{Z}^{+}$so that $f^{\tau} p=p$. This periodic orbit $\mathcal{O}_{f}(p)=\left\{p, f p, \ldots, f^{\tau-1} p\right\}$ is attracting (a sink, for short) if there exists a neighborhood $V_{p}$ of $p$ such that $\left.f^{\tau}\right|_{V_{p}}: V_{p} \rightarrow V_{p}$ is a contraction: there exists $0<\lambda<1$ so that $\operatorname{dist}\left(f^{\tau} q, f^{\tau} r\right)<\lambda \operatorname{dist}(q, r), \forall q, r \in V_{p}$. Equivalently, $\left\|D f^{\tau}(p)\right\|<\lambda$ for some $\lambda \in(0,1)$. 

With this set up, V. Ara\'ujo proves that a measure with negative Lyapunov exponents is concentrated on a sink:

\begin{corollary} \cite{A20}\label{cap1}
Let $\mu$ be an invariant probability measure with respect to a $C^{1} \operatorname{map} f: M \rightarrow$ $M$ such that $\inf _{x \in M}\|D f(x)\|>0$ and $\chi(x)<0, \mu$-a.e. $x \in M$. Then $\mu$ decomposes as $\tilde{\mu}+\sum_{i \geq 1} \mu_{i}$, where each $\mu_{i}$ is a Dirac mass equidistributed on a periodic attracting orbit of $f$ (a sink), the sum is over at most countably many such orbits, and $\tilde{\mu}$ (which might be the null measure) satisfies $A_{1}^{-}(x) \geq 0, \tilde{\mu}$-a.e. $x \in M$. In addition, if $\mu$ is $f$-ergodic, then $\mu$ is concentrated on the orbit of a periodic attractor (sink).
\end{corollary} 

\

Now consider a given $C^{1}-$ocal diffeomorphism on the circle $f: \Sc^1 \to \Sc^1$ and an $f$-ergodic measure $\mu$ with
negative Lyapunov exponent $\chi_{\mu}(f) < 0$. Given the restriction on the dimension of $M$, in order to apply Corollary \ref{cap1}, it is sufficient to consider the product map $f\times f$ acting on $\Sc^1 \times \Sc^1$ and the ergodic measure $\mu \times \mu$, with $\chi_{\mu\times\mu}(f\times f)=\chi_\mu(f)<0$. We conclude that $\mu \times \mu$ is concentrated on the orbit of a sink, then so is $\mu$. To sum up, an ergodic measure with negative Lyapunov exponent for a $C^1$ map of the circle is necessarily supported on an attracting orbit.

\end{appendix}

\bibliographystyle{alpha}

\end{document}